\documentclass{article}
\usepackage{graphicx,indentfirst}
\usepackage{amscd}

\usepackage[all]{xy}
\usepackage{amsmath}
\usepackage{amsfonts,amssymb}
\usepackage{extarrows}
\usepackage{amsthm}
\usepackage{latexsym}
\usepackage[american]{babel}
\usepackage{times}
\usepackage{tikz}
\usepackage{tikz-cd}
\usepackage{url}
\usepackage{verbatim}
\usepackage{enumerate}
\usepackage{hyperref}

\usetikzlibrary{matrix,arrows,decorations.pathmorphing}
\usepackage[a4paper,left=2.5cm,right=2.5cm,bottom=3cm,top=3.5cm]{geometry}

\makeatletter
\newcommand*{\relrelbarsep}{.386ex}
\newcommand*{\relrelbar}{%
  \mathrel{%
    \mathpalette\@relrelbar\relrelbarsep
  }%
}
\newcommand*{\@relrelbar}[2]{%
  \raise#2\hbox to 0pt{$\m@th#1\relbar$\hss}%
  \lower#2\hbox{$\m@th#1\relbar$}%
}
\providecommand*{\rightrightarrowsfill@}{%
  \arrowfill@\relrelbar\relrelbar\rightrightarrows
}
\providecommand*{\leftleftarrowsfill@}{%
  \arrowfill@\leftleftarrows\relrelbar\relrelbar
}
\providecommand*{\xrightrightarrows}[2][]{%
  \ext@arrow 0359\rightrightarrowsfill@{#1}{#2}%
}
\providecommand*{\xleftleftarrows}[2][]{%
  \ext@arrow 3095\leftleftarrowsfill@{#1}{#2}%
}
\makeatother

\RequirePackage[T1]{fontenc}

\DeclareMathOperator*{\holim}{holim}
\DeclareMathOperator*{\colim}{colim}
\DeclareMathOperator*{\hocolim}{hocolim}

\begin{document}
\newcommand{\End}{\text{End}}
\newcommand{\ad}{\text{ad}}
\newcommand{\tr}{\text{tr}}
\newcommand{\bDelta}{\bf{\Delta}}
\newcommand{\Hom}{\text{Hom}}
\newcommand{\D}{\text{D}}
\newcommand{\Tot}{\text{Tot}}
\newcommand{\dgCat}{\text{dgCat}}
\newcommand{\DK}{\text{DK}}
\newcommand{\cat}{\mathcal}

\theoremstyle{plain}
\newtheorem{thm}{Theorem}
\newtheorem{prop}{Proposition}
\newtheorem{coro}{Corollary}
\newtheorem{conj}{Conjecture}
\newtheorem{lemma}{Lemma}

\theoremstyle{definition}
\newtheorem{defi}{Definition}[section]

\theoremstyle{remark}
\newtheorem{eg}{Example}[section]
\newtheorem{rmk}{Remark}[section]
\newtheorem{ctn}{Caution}

\title{Explicit homotopy limits of dg-categories and twisted complexes}
\author{Jonathan Block,
\textit{blockj@math.upenn.edu},\\
\small{Department of Mathematics,
209 S 33 Street,
University of Pennsylvania,
Philadelphia,
PA, 19104
USA}
\and Julian Holstein, \textit{julianvsholstein@gmail.com},\\
\small{Department of Mathematics and Statistics,
Fylde College,
Lancaster University,
Lancaster, LA1 4YF,
UK}
\and
Zhaoting Wei,
\textit{zwei3@kent.edu},\\
\small{Department of Mathematical Sciences,
Kent State University,
Kent,
OH 44242,
USA}}

\maketitle

MSC Classes: 18D20, 18G55, 14F05


\begin{abstract}
In this paper we study the homotopy limits of cosimplicial diagrams of dg-categories.
We first give an explicit construction of the totalization of such a diagram and then show that the totalization agrees with the homotopy limit in the following two cases:
(1) the complexes of sheaves of $\mathcal O$-modules on the \v{C}ech nerve of an open cover of a ringed space $(X, \mathcal O)$;
(2) the complexes of sheaves on the simplicial nerve of a discrete group $G$ acting on a space.
The explicit models we obtain in this way are twisted complexes as well as their $D$-module and $G$-equivariant versions. As an application we show that there is a stack of twisted perfect complexes.
\end{abstract}

\section{Introduction}
Homotopy limits in a model category are an important way to present descent data. In this paper we consider homotopy limits in $\dgCat_{\DK}$,
the category of (small) dg-categories with the Dwyer-Kan model structure \cite{tabuada2005structure}.

The homotopy limit of some diagrams in $\dgCat_{\DK}$ can be explicitly constructed. For example in \cite{ben2013milnor} Section 4 the \emph{homotopy fibre product} $B\times^h_D C$ of the diagram
$$
\begin{CD}
 @. B\\
@.@VVcV\\
C @>d>> D
\end{CD}
$$
is given using the \emph{path object} in $\dgCat_{\DK}$ constructed in \cite{tabuada2010homotopy}. In \cite{ben2013milnor} the authors use this construction to further prove the Milnor descent of cohesive dg-categories.

Higher dimensional analogues of the path object are given by \emph{simplicial resolutions} in $\dgCat_{\DK}$ which were constructed in (\cite{holstein2014properness}).
They can be used to explicitly compute arbitrary homotopy limits and we will use them to take homotopy limits over cosimplicial diagrams.

In algebra and geometry it is often interesting to consider the homotopy limit of a cosimplicial diagram in $\dgCat$.
For example, let $\mathcal{U}=\{U_i\}$ be an open cover of a ringed space $(X, \mathcal{O})$ and the functor
$$
Cpx: \text{Space}^{op}\rightarrow \dgCat
$$
be the contravariant functor which assigns to each  ringed space $U$ the dg-category of complexes of $\mathcal{O}_U$-modules on $U$. Then we have a cosimplicial diagram in $\dgCat_{\DK}$
\begin{equation}\label{equation: cosimplicial diagram of Cpx of an open cover}
\begin{tikzcd}
\prod Cpx(U_i) \arrow[yshift=0.7ex]{r}\arrow[yshift=-0.7ex]{r}& \prod Cpx(U_i\cap U_j) \arrow[yshift=1ex]{r}\arrow{r}\arrow[yshift=-1ex]{r}  &  \prod Cpx(U_i\cap U_j\cap U_k) \cdots
\end{tikzcd}
\end{equation}
Intuitively the homotopy limit of Diagram \ref{equation: cosimplicial diagram of Cpx of an open cover} is given by gluing data and higher gluing data and in this paper we give a precise construction: the homotopy limit is given by the dg-category of \emph{twisted complexes}, see Section \ref{subsection: twisted complexes} for details.

We deduce from this characterization that the presheaf of twisted perfect complexes is in fact a stack.
\begin{rmk}
 In \cite{toledo1978duality} Toledo and Tong  first introduced twisted complexes as a tool to glue the characteristic classes of coherent sheaves on complex manifolds. Implicitly they have shown that twisted complexes present the descent data of complexes of sheaves. However, as far as we know, the fact that twisted complexes are the homotopy limit of Diagram (\ref{equation: cosimplicial diagram of Cpx of an open cover})  has never been explicitly stated and proved  in the literature. This is the main reason for us to write the current paper.
\end{rmk}

\begin{rmk}
We will also consider variations of this question where the cover is not given by open subsets of $X$ and $Cpx$ is only a pseudo-functor
instead of a functor and hence the diagram corresponding to Diagram (\ref{equation: cosimplicial diagram of Cpx of an open cover}) is not a strict cosimplicial diagram. We resolve this problem by rectification, see Section \ref{subsection: rectification} for details.
\end{rmk}

The strategy of our construction is that we first construct the \emph{totalization} of the cosimplicial diagram
and then show that the homotopy limit and totalization are weakly equivalent in good cases.
Actually many interesting cosimplicial diagrams in $\dgCat$ fall into these "good cases", see Section \ref{subsection: twisted complexes} and Section \ref{subsection: group action} for details.

The paper is organized as follows: In Section \ref{section: A review of homotopy limit in a model category} we review the Reedy model structure and the homotopy limit of a cosimplicial diagram in a general model category.

In Section \ref{section: The homotopy limits of cosimplicial objects in dgCat} we focus on dg-categories. In Section \ref{subsection: simplicial resolution in dgCat} we review the simplicial resolution of a dg-category. In Section \ref{subsection: homotopy limit in dgCat} we give the construction of the totalization and homotopy limit of a cosimplicial diagram in $\dgCat_{\DK}$.

In Section \ref{section: twisted complexes} we use these results to study twisted complexes. In
particular, in Section \ref{subsection: rectification} we construct a rectification of the diagrams associated with pseudo-functors. In Section \ref{subsection: Criterion of Reedy fibrancy} we give a criterion for when the totalization is weakly equivalent to the homotopy limit. In Section \ref{subsection: twisted complexes} we study the homotopy limit of Diagram \ref{equation: cosimplicial diagram of Cpx of an open cover}, i.e. complexes of sheaves on the \u{C}ech nerve of an open cover, and show this is equivalent to the category of twisted complexes. In Section \ref{subsection: twisted perfect complexes} we study the homotopy limit of strict perfect complexes with respect to open covers. In Section \ref{subsection: stack of twisted complexes} we construct a presheaf of twisted perfect complexes and show that it is in fact a stack.
In Section \ref{subsection: group action} we study the complexes of sheaves on the simplicial space associated to a group acting on a space.

\section*{Acknowledgement}
We would like to thank David Ben-Zvi, Ezra Getzler, Valery Lunts, Tom Nevins, Ajay Ramadoss, Olaf Schn\"{u}rer, Junwu Tu and Shizhuo Zhang for  helpful discussions.

We would also like to thank the referee for useful comments.

\section{A review of homotopy limits in a model category}\label{section: A review of homotopy limit in a model category}
\subsection{A quick review of the Reedy model structure}\label{subsection: Reedy model structure}
In this subsection we give a quick review of Reedy model structure. For more details
 see \cite{hirschhorn2009model} Chapter 15 or \cite{riehl2014categorical} Chapter 14.

 \begin{defi}\label{defi: Reedy category}[\cite{hirschhorn2009model} Definition 15.1.2]
 A \emph{Reedy category} is a (small) category $\mathcal{C}$ equipped with two subcategories $\overrightarrow{\mathcal{C}}$ (the direct category) and $\overleftarrow{\mathcal{C}}$ (the inverse category), both of which contain all the objects of $\mathcal{C}$, in which every object can be assigned a nonnegative integer (called the degree) such that
 \begin{enumerate}
 \item Every non-identity morphism of $\overrightarrow{\mathcal{C}}$ raises degree.
 \item Every non-identity morphism of $\overleftarrow{\mathcal{C}}$ lowers degree.
 \item Every morphism $g$ in $\mathcal{C}$ has a unique factorization $g=\overrightarrow{g}\overleftarrow{g}$ where $\overrightarrow{g}$ is in $\overrightarrow{\mathcal{C}}$  and $\overleftarrow{g}$ is in $\overleftarrow{\mathcal{C}}$.
 \end{enumerate}
 \end{defi}

 We need the  concepts of latching and matching categories as follows.
 \begin{defi}\label{defi: latching and matching categories}[\cite{hirschhorn2009model} Definition 15.2.3]
 Let $\mathcal{C}$ be a Reedy category and let $\alpha$ be an object of $\mathcal{C}$.
The \emph{latching category} $\partial(\overrightarrow{\mathcal{C}}\downarrow \alpha)$ of $\mathcal{C}$ at $\alpha$ is the full subcategory of the overcategory $(\overrightarrow{\mathcal{C}}\downarrow \alpha)$ containing all the objects except the identity map of $\alpha$.
The \emph{matching category} is defined dually.
 \end{defi}

 Now let $\mathcal{C}$ be a Reedy category and $\mathcal{M}$ a category closed under limits and colimits. We consider the category $\mathcal{M}^{\mathcal{C}}$ of $\mathcal{C}$-diagrams in $\mathcal{M}$.

 \begin{defi}\label{defi: latching and matching objects}[\cite{hirschhorn2009model} Definition 15.2.5]
Let $\mathcal{C}$ be a Reedy category and $\mathcal{M}$ a category closed under limits and colimits. Let $\mathcal{X}$ be a $\mathcal{C}$-diagram in $\mathcal{M}$, i.e. $\mathcal{X}$ is a functor $\mathcal{C}\to \mathcal{M}$. Let $\alpha$ be an object of $\mathcal{C}$.
The \emph{latching object} of $\mathcal{X}$ at $\alpha$ is $L_{\alpha}\mathcal{X}=\colim_{\partial(\overrightarrow{\mathcal{C}}\downarrow \alpha)}\mathcal{X}$ and the latching map of $\mathcal{X}$ at $\alpha$ is the natural map $L_{\alpha}\mathcal{X}\to \mathcal{X}_{\alpha}$.
The \emph{matching object} $M_{\alpha}X$ is defined dually.
 \end{defi}

 Let $\mathcal{X}$ and $\mathcal{Y}$ be  $\mathcal{C}$-diagrams in $\mathcal{M}$ and $f: \mathcal{X}\to \mathcal{Y}$ be a map of $\mathcal{C}$-diagrams. It is clear that $f$ induces maps between latching and matching objects $f: L_{\alpha}\mathcal{X}\to L_{\alpha}\mathcal{Y}$ and $f: M_{\alpha}\mathcal{X}\to M_{\alpha}\mathcal{Y}$. Moreover we can define the following maps.

 \begin{defi}\label{defi: relative latching and matching maps}[\cite{hirschhorn2009model} Definition 15.3.2]
Let $\mathcal{C}$ be a Reedy category and $\mathcal{M}$ a category closed under limits and colimits. Let
$f: \mathcal{X}\to \mathcal{Y}$ be a map of $\mathcal{C}$-diagrams. Let $\alpha$ be an object of $\mathcal{C}$.
The \emph{relative latching map} of $f$ at $\alpha$ is the map $\mathcal{X}_{\alpha}\coprod_{L_{\alpha}\mathcal{X}}L_{\alpha}\mathcal{Y}\to \mathcal{Y}_{\alpha}$.
The \emph{relative matching map} is defined dually.
 \end{defi}

 \begin{defi}\label{defi: Reedy model structure}[\cite{hirschhorn2009model} Definition 15.3.3]
 Let $\mathcal{C}$ be a Reedy category and $\mathcal{M}$ a model category.  Let
$f: \mathcal{X}\to \mathcal{Y}$ be a map of $\mathcal{C}$-diagrams, then
\begin{enumerate}
\item $f$ is a \emph{Reedy weak equivalence} if for every object $\alpha$ of $\mathcal{C}$, the map $f_{\alpha}: \mathcal{X}_{\alpha}\to \mathcal{Y}_{\alpha}$ is a weak equivalence in $\mathcal{M}$.
\item  $f$ is a \emph{Reedy cofibration} if for every object $\alpha$ of $\mathcal{C}$, the  relative latching map $\mathcal{X}_{\alpha}\coprod_{L_{\alpha}\mathcal{X}}L_{\alpha}\mathcal{Y}\to \mathcal{Y}_{\alpha}$ is a cofibration in $\mathcal{M}$.
\item  $f$ is a \emph{Reedy fibration} if for every object $\alpha$ of $\mathcal{C}$, the  relative matching map $\mathcal{X}_{\alpha}\to  \mathcal{Y}_{\alpha}\times_{M_{\alpha}\mathcal{Y}} M_{\alpha}\mathcal{X}$ is a fibration in $\mathcal{M}$.
\end{enumerate}
This defines a model structure on $\mathcal{M}^{\mathcal{C}}$, called  the \emph{Reedy model category structure}.
 \end{defi}

 In particular, a $\mathcal{C}$-diagram $\mathcal{X}$ is Reedy fibrant if for every object $\alpha$ of $\mathcal{C}$, the matching map $\mathcal{X}_{\alpha}\to M_{\alpha}\mathcal{X}$ is a fibration in $\mathcal{M}$.

\subsection{A review of homotopy limits of general Reedy diagrams in a model category}\label{subsection: review of homotopy limit in general}
In   this subsection we review the construction of homotopy limit in a model category. For more details see \cite{hirschhorn2009model} Chapter 18 and 19.

Let $\mathcal{C}$ be a small category, then for any object $c$ of $\mathcal{C}$ we can consider the overcategory $(\mathcal{C}\downarrow c)$ and its nerve $N(\mathcal{C}\downarrow c)$. Moreover a morphism $\sigma: c\to c^{\prime}$ in $\mathcal{C}$ induces a functor $\sigma_*: (\mathcal{C}\downarrow c)\to (\mathcal{C}\downarrow c^{\prime})$ and hence a map between simplicial sets $N(\sigma_*): N(\mathcal{C}\downarrow c)\to N(\mathcal{C}\downarrow c^{\prime})$.
In conclusion $N(\mathcal{C}\downarrow -)$ gives a $\mathcal{C}$-diagram of simplicial sets.

Now we consider a model category $\mathcal{M}$ and let $\mathcal{X}$ be a $\mathcal{C}$-diagram in $\mathcal{M}$. We want to explicitly construct the homotopy limit $\holim \mathcal{X}$. First we introduce the following definition.

\begin{defi}[\cite{hirschhorn2009model} Definition 16.6.1]\label{defi: simplicial frame}
Let $X$ be an object in a model category $\mathcal{M}$. A \emph{simplicial frame} on $X$ is a simplicial object $\widehat{X}$ in $\mathcal{M}$ together with a weak equivalence $cs_*X\to \widehat{X}$ in the Reedy model category structure on $\mathcal{M}^{\Delta^{op}}$ such that
\begin{enumerate}
\item the induced map $X\to \widehat{X}_0$ is an isomorphism;
\item if $X$ is a fibrant object of $\mathcal{M}$, then $\widehat{X}$ is a fibrant object in the Reedy model category $\mathcal{M}^{\Delta^{op}}$.
\end{enumerate}
\end{defi}

If $\mathcal{Y}=Y_{\bullet}$ is a simplicial object in $\mathcal{M}$ and $K$ is a simplicial set, then we can naturally define the mapping object $\mathcal{Y}^K$. It is clear that the construction is functorial in both $\mathcal{Y}$ and $K$.

Now we can give the formula of homotopy limits.

\begin{defi}[\cite{hirschhorn2009model} Definition 19.1.5]\label{defi: homotopy limit explicit formula general form}
Let $\mathcal{M}$ be a framed model category and $\mathcal{C}$ be a small category. If $\mathcal{X}$ is a $\mathcal{C}$-diagram in $\mathcal{M}$, then the \emph{homotopy limit}  $\holim(\mathcal{X})$ of $\mathcal{X}$ is defined to be the equalizer of the maps
\begin{equation}\label{equation: formula of homotopy limit in general}
\prod_{c\in Ob(\mathcal{C})}(\widehat{X_{c}})^{N(\mathcal{C}\downarrow c)}\underset{\psi}{\overset{\phi}{\rightrightarrows}}\prod_{(\sigma:c\to c^{\prime})\in Mor(\mathcal{C})}(\widehat{X_{c^{\prime}}})^{N(\mathcal{C}\downarrow c)}
\end{equation}
where $\widehat{X_{c}}$ is the natural simplicial frame on $X_{c}$ and $N(\mathcal{C}\downarrow c)$ is the nerve of the over category $\mathcal{C}\downarrow c$. The projection of the map $\phi$ on the factor $\sigma: c\to c^{\prime}$ is the composition of a natural projection from the product with the map
$$
\sigma_*^{id_{N(\mathcal{C}\downarrow c)}}:(\widehat{X_{c}})^{N(\mathcal{C}\downarrow c)}\to (\widehat{X_{c^{\prime}}})^{N(\mathcal{C}\downarrow c)}
$$
and the projection of the map $\psi$ on the factor $\sigma: c\to c^{\prime}$ is the composition of a natural projection from the product with the map
$$
(id_{\widehat{X_{c^{\prime}}}})^{N(\sigma_*)}: (\widehat{X_{c^{\prime}}})^{N(\mathcal{C}\downarrow c^{\prime})}\to (\widehat{X_{c^{\prime}}})^{N(\mathcal{C}\downarrow c)}.
$$
\end{defi}

\subsection{Homotopy limit and the totalization of a cosimplicial object}\label{subsection: review of homotopy limit of cosimplicial diagram}
In this subsection we take $\mathcal{C}$ to be the cosimplicial index category $\mathbf{\Delta}$. In this case a $\mathbf{\Delta}$-diagram in $\mathcal{M}$ is exactly a cosimplicial object in $\mathcal{M}$. Let $\mathcal{X}=\{X^{\bullet}\}$ be such a cosimplicial object. By Definition \ref{defi: homotopy limit explicit formula general form} the homotopy limit of $\mathcal{X}$ has the following form
\begin{equation}\label{equation: formula of homotopy limit of cosimplicial object in general}
\prod_{[n]\in Ob(\mathbf{\Delta})}(X^n)^{N({\bDelta\downarrow} [n])}\underset{\psi}{\overset{\phi}{\rightrightarrows}}\prod_{(\sigma:[n]\to [m])}(X^m)^{N({\bDelta\downarrow} [n])}.
\end{equation}

Although it is possible to compute the homotopy limit by Equation (\ref{equation: formula of homotopy limit of cosimplicial object in general}), we have a problem that the simplicial set $N({\bDelta\downarrow} [n])$ is too big. Hence we are looking for an alternative formula of the homotopy limit.

\begin{defi}[\cite{hirschhorn2009model} Definition 18.6.3 and 19.8.1 or \cite{riehl2014categorical} Example 4.3.1 ]\label{defi: total object of cosimplicial object}
Let $\mathcal{X}=\{X^{\bullet}\}$ be   a cosimplicial object in a framed model category $\mathcal{M}$. The \emph{totalization} of $\mathcal{X}$ $\Tot(\mathcal{X})$ is the equalizer of the maps
\begin{equation}\label{equation: formula of totalization in general}
\prod_{[n]\in Ob(\mathbf{\Delta})}(X^n)^{\Delta[n]}\underset{\psi}{\overset{\phi}{\rightrightarrows}}\prod_{(\sigma:[n]\to [m])}(X^n)^{\Delta[n]}
\end{equation}
where the projection of the map $\phi$ on the factor $\sigma: [n]\to[m]$ is the composition of the natural projection from the product with the map
$
\Delta_n(\sigma_*): (X^n)^{\Delta[n]}\to (X^m)^{\Delta[n]}
$
and the projection of the map $\psi$ on the factor $\sigma: [n]\to[m]$ is the composition of the natural projection from the product with the map
$
\sigma^*(X^m): (X^m)^{\Delta[m]}\to (X^m)^{\Delta[n]}.
$
\end{defi}

We want to seek  relations between totalization and homotopy limit. For Reedy fibrant cosimplicial objects we have the following theorem.

\begin{thm}\label{thm: totalization and homotopy limit are weakly equi for fibrant object}
Let $\mathcal{X}=\{X^{\bullet}\}$ be a \emph{Reedy fibrant}    cosimplicial object in a framed model category $\mathcal{M}$. Then $\Tot(\mathcal{X})$ is naturally weakly equivalent to $\holim(\mathcal{X})$.
\end{thm}
\begin{proof}
\cite{hirschhorn2009model} Theorem 19.8.4.
\end{proof}

\begin{rmk}
In fact in the original statement of \cite{hirschhorn2009model} Theorem 19.8.4 there is no requirement of Reedy fibrancy. However it seems to be a typo.
\end{rmk}

Theorem \ref{thm: totalization and homotopy limit are weakly equi for fibrant object} tells us that in good cases we can use the totalization to compute the homotopy limit, which will dramatically simplify the computation.

\section{The homotopy limits of cosimplicial objects in $\dgCat_{\DK}$}\label{section: The homotopy limits of cosimplicial objects in dgCat}
In this section we take $\mathcal{M}$ to be the $\dgCat_{\DK}$, i.e. the category of (small) dg-categories equipped with the Dwyer-Kan model structure \cite{tabuada2005structure}.

\begin{defi}\label{defi: DK model structure}
Let dgCat be the category of small dg-categories. There is a model structure on dgCat, called the \emph{Dwyer-Kan (DK) model structure}, where a dg-functor $F: \mathcal{B}\to \mathcal{C}$ is
\begin{itemize}
\item a weak equivalence if
\begin{enumerate}
\item for all objects $x, y\in \mathcal{B}$, the component $F_{x,y}: \mathcal{B}(x,y)\to \mathcal{C}(F(x),F(y))$ is a quasi-isomorphism of complexes;
\item the induced functor $H^0(F): H^0(\mathcal{B}) \to H^0(\mathcal{C})$p is an equivalence of categories.
\end{enumerate}
\item a fibration if
\begin{enumerate}
\item for all objects $x, y\in \mathcal{B}$, the component $F_{x,y}$ is a degreewise surjection;
\item for each morphism $\theta: F(x)\to z$  which is invertible in $H^0(\mathcal{C})$, it could be lifted to an morphism $\widetilde{\theta}$ which is invertible in $H^0(\mathcal{B})$.
\end{enumerate}
\end{itemize}
We denote dgCat with the Dwyer-Kan model structure by $\text{dgCat}_{\text{DK}}$.
\end{defi}

We will give an explicit construction of the totalization and homotopy limit of a cosimplicial object in $\dgCat_{\DK}$.

\begin{rmk}
There is another model structure on dgCat introduced in \cite{tabuada2005invariants} called the \emph{Morita model structure}. This is the left Bousfield localization of the DK model structure along functors which induce equivalences on the derived categories. Here the derived category of a dg-category $\cat D$ is the category of functors from the opposite of $\cat D$ to chain complexes, localized at objectwise quasi-isomorphisms.

It is important to note that any computation of homotopy limits in the DK model structure also allows the computation of homotopy limits in the Morita model structure, see \cite{holstein2015moritacohomology} \mbox{Lemma 1}.
 \end{rmk}

\subsection{The simplicial resolution in the category of dg-categories}\label{subsection: simplicial resolution in dgCat}
In this subsection we discuss the simplicial resolution in the category of dg-categories, which is introduced in \cite{holstein2014properness}. First we introduce some notations and definitions.

Let $\mathcal{B}$ be a (small) dg-category and $S$ be a (fixed) set of index and $\mathcal{E}=\{E_i|i\in S\}$ and $\mathcal{F}=\{F_i|i\in S\}$ be two collections of objects of $\mathcal{B}$.  We define
\begin{equation}
C^p(\mathcal{B},\Hom^q(\mathcal{E},\mathcal{F}))=\prod_{|I|=p+1,\\ I=\{i_0\ldots i_p\}}\Hom^q_{\mathcal{B}}(E_{i_p},F_{i_0}).
\end{equation}
For an index set $I=\{i_0,\ldots i_p\}$ (the $i_j$'s may be repeated), an element $\phi^{p,q}_I=\phi^{p,q}_{i_0\ldots i_p}$ of $C^p(\mathcal{B},\Hom^q(\mathcal{E},\mathcal{F}))$ is an element in $\Hom^q_{\mathcal{B}}(E_{i_p},F_{i_0})$.
Here we call $p$ the simplicial degree of $\phi$, call $q$ the $\mathcal{B}$ degree of $\phi$ and call $|\phi|=p+q$ the total degree of $\phi$.

\begin{defi}\label{defi: differential delta and multiplication in simplicial dg-cat}
The differential on $\phi\in \prod_{p+q=k}C^p(\mathcal{B},\Hom^q(\mathcal{E},\mathcal{F}))$ is defined by the following formula.
\begin{equation}\label{equation: D in simplicial dg}
(\D\phi)^{p,q+1}_{i_0\ldots i_p}:=(-1)^pd_{\mathcal{B}}(\phi^{p,q}_{i_0\ldots i_p})+\sum_{j=1}^{p-1}(-1)^j\phi^{p-1,q+1}_{i_0\ldots \widehat{i_j}\ldots i_p}.
\end{equation}
It is easy to check that $\D\circ \D=0$.

Moreover, we define the \emph{shuffled multiplication} of two elements $\phi\in C^p(\mathcal{B},\Hom^q(\mathcal{F},\mathcal{G}))$ and $\eta\in C^r(\mathcal{B},\Hom^s(\mathcal{E},\mathcal{F}))$ as
\begin{equation}\label{equation: shuffled multiplication in simplicial dg}
(\phi\cdot \eta)^{p+r,q+s}_{i_0\ldots i_{p+r}}=(-1)^{qr}\phi^{p,q}_{i_0\ldots i_p}\circ \eta^{r,s}_{i_p\ldots i_{p+r}}
\end{equation}
From now on the notation "$\cdot$" will be preserved for the shuffled multiplication.

We can also check that the differential and multiplication satisfy the Leibniz rule, i.e.
$$
\D(\phi\cdot \eta)=(\D\phi)\cdot \eta+(-1)^{|\phi|}\phi\cdot (\D\eta).
$$
\end{defi}

\begin{defi}\label{defi: n simplex in dg-category}[See \cite{holstein2014properness} Definition 3.2]
Let $\mathcal{B}$ be a (small) dg-category. Then $\Delta_n(\mathcal{B})$ is a dg-category with objects given by pairs $(\mathcal{E}, \phi) $
where $\mathcal{E}$ is a collection   $E_0,E_1,\ldots, E_n$ of objects in $\mathcal{B}$  and $\phi$ is a collection $\phi^{k,1-k}_I\in \Hom^{1-k}_{\mathcal{B}}(E_{i_k},E_{i_0})$ for all multi-indices $I=(i_0,\ldots ,i_k)$, $0\leq i_0\leq \ldots \leq i_k\leq n$. These pairs must satisfy the Maurer-Cartan equation
\begin{equation}
\D\phi+\phi\cdot\phi=0.
\end{equation}
More precisely, for any $I=\{i_0\ldots i_k\}$, we have
\begin{equation}\label{equation: MC for simplicial resolution explicit}
(-1)^kd_{\mathcal{B}}(\phi^{k,1-k}_{i_0\ldots i_k})+\sum_{j=1}^{k-1}(-1)^j\phi^{k-1,2-k}_{i_0\ldots \widehat{i_j}\ldots i_k}+\sum_{j=1}^{k-1}(-1)^{(1-j)(k-j)}\phi^{j,1-j}_{i_0\ldots i_j}\circ \phi^{k-j,1-k+j}_{i_j\ldots i_k}=0.
\end{equation}

Moreover we require that all $\phi^{1,0}_{ij}\in \Hom^0_{\mathcal{B}}(E_i,E_j)$ (including $\phi^{1,0}_{ii}$) are homotopically invertible and in particular $\phi^{1,0}_{ii}=id$. On the other hand we require all   $\phi^{k,1-k}_{i_0\ldots i_k}$'s equal to $0$ if they have   repeated indices and  $k\geq 2$.

Morphisms from $(\mathcal{E}, \phi)$ to $(\mathcal{F}, \psi)$ are given as follows.
$$
\Hom^m_{\Delta_n(\mathcal{B})}((\mathcal{E}, \phi),(\mathcal{F}, \psi))=\prod_{k}\prod_{\substack{(i_0\ldots i_k)\\0\leq i_0\leq\ldots \leq i_k\leq n}}\Hom^{m-k}_{\mathcal{B}}(E_{i_k},F_{i_0})
$$
and the differential $d$ on a degree $m$ morphism $\theta$ is given by
\begin{equation}\label{equation: differential on morphisms of simplicial resolution}
d\theta=\D\theta+\psi\cdot \theta-(-1)^m\theta\cdot \phi.
\end{equation}
\end{defi}

\begin{rmk}
Definition \ref{defi: n simplex in dg-category} differs   from the simplicial resolution in \cite{holstein2014properness} in the sense that in Definition \ref{defi: n simplex in dg-category} we allow the $\phi^{k,1-k}$'s  with repeated subscripts, while the simplicial resolution in \cite{holstein2014properness} requires strictly increasing subscripts. However, we demand that for $k\geq 2$,  all $\phi^{k,1-k}$'s  with repeated subscripts are $0$, thus the two definition are given by exactly the same data. The differential also agrees.
Moreover, in Definition \ref{defi: n simplex in dg-category} we also make the $\pm$ sign coincide with those in \cite{o1981trace}
\end{rmk}

\begin{eg}\label{eg: Delta1 of a dg-category}
An object in $\Delta_1(\mathcal{B})$ consists of two object $E_0$ and $E_1$ of $\mathcal{B}$, a degree $0$  morphism $\phi^{1,0}_{01}: E_0\to E_1$ which is invertible in the homotopy category together with the $\phi^{1,0}_{00}$ and $\phi^{1,0}_{11}$.

A degree $m$ morphism from $(\mathcal{E},\phi)$ to $(\mathcal{F},\psi)$ is given by a collection of morphisms $\theta^{0,m}_{0}$, $\theta^{0,m}_1$, $\theta^{1,m-1}_{01}$ and the degenerate ones $\theta^{1,m-1}_{00}$ etc. The differential on $\theta$ is given by Equation (\ref{equation: differential on morphisms of simplicial resolution}). In particular the differential of the non-degenerate $\theta$'s is given by
$$
D(\theta^{0,m}_{0}, \theta^{0,m}_1, \theta^{1,m-1}_{01})=(d\theta^{0,m}_{0}, d\theta^{0,m}_1, d\theta^{1,m-1}_{01}+\psi^{1,0}_{01}\theta^{0,m}_1-(-1)^m\theta^{0,m}_0 \phi^{1,0}_{01})
$$
\end{eg}

\begin{rmk}
The non-degenerate part of $\Delta_1(\mathcal{B})$ agrees with the path object  $P(\mathcal{B})$ as constructed in \cite{tabuada2010homotopy} Section 3. The only differences is that we include two morphisms $\phi^{1,0}_{00}$ and $\phi^{1,0}_{11}$.
\end{rmk}

\begin{defi}\label{defi: simplicial maps between the n-simplexes in a dg-category}
For any   map $\sigma: [n]\to [m]$ in the simplicial category $\bf{\Delta}$, we can define the induced map $\sigma^*: \Delta_m(\mathcal{B})\to \Delta_n(\mathcal{B})$. Actually for an object $(\mathcal{E}=\{E_0,\ldots, E_m\},\phi)$ of $\Delta_m(\mathcal{B})$, we define
$$
\sigma_*(\mathcal{E})=\{E_{\sigma(0)},\ldots, E_{\sigma(n)}\}
$$
and
$$
(\sigma_*(\phi))^{k,1-k}_{i_0\ldots i_k}=\phi^{k,1-k}_{\sigma(i_0)\ldots \sigma(i_k)}.
$$

$\sigma_*$ of morphisms are defined similarly.
\end{defi}

Therefore $\Delta_{\bullet}(\mathcal{B})$ is a cosimplicial object in $\dgCat$. In \cite{holstein2014properness} it has been proved that  $\Delta_{\bullet}(\mathcal{B})$ has the following properties.

\begin{prop}[\cite{holstein2014properness} Proposition 3.9]\label{prop: inclusion of the constant simplicial is a weak equivalence}
The inclusion from the constant simplicial dg-category $c\mathcal{B}$ to $\Delta_{\bullet}(\mathcal{B})$ is a levelwise weak equivalence.
\end{prop}
\begin{proof}
See \cite{holstein2014properness} Proposition 3.9.
\end{proof}

The next important property of  $\Delta_{\bullet}(\mathcal{B})$ is the Reedy fibrancy.

\begin{prop}[\cite{holstein2014properness} Proposition 3.10]\label{prop: simplicial resolution is Reedy fibrant}
The simplicial dg-category $\Delta_{\bullet}(\mathcal{B})$ is Reedy fibrant.
\end{prop}
\begin{proof}
See \cite{holstein2014properness} Proposition 3.10.
\end{proof}

\begin{coro}\label{coro: simplicial resolution in dg-cat is a simplicial frame}
For a dg-category $\mathcal{B}$, the simplicial resolution $\Delta_{\bullet}(\mathcal{B})$ in Definition \ref{defi: n simplex in dg-category} give a simplicial frame of $B$ in $\dgCat_{\DK}$.
\end{coro}
\begin{proof}
It is a direct corollary of Definition \ref{defi: simplicial frame}, Proposition \ref{prop: inclusion of the constant simplicial is a weak equivalence} and Proposition \ref{prop: simplicial resolution is Reedy fibrant}.
\end{proof}

\begin{rmk}A related concept is the category $dga_{\geq 0}$ consisting of dg-algebras of nonnegative degree and there is a projective model structure on $dga_{\geq 0}$ such that the inclusion $dga_{\geq 0}\hookrightarrow dgCat$  preserves the model structure.  For any dg-algebra $B$ of nonnegative degree there is a simplicial resolution of $B$ given by $\Omega_{\text{poly}}(\Delta^{\bullet})\otimes B$, where $\Omega_{\text{poly}}(\Delta^n)$ is the ag-algebra of the polynomial differential forms on the geometric simplex $\Delta^n$.

However, for a dg-category $\mathcal{B}$, $\Omega_{\text{poly}}(\Delta^{\bullet})\otimes \mathcal{B}$ in general is not Reedy fibrant hence does not give a simplicial resolution. In fact the natural dg-functor $p: \Omega_{\text{poly}}(\Delta^1)\otimes \mathcal{B}\to \Omega_{\text{poly}}(\partial \Delta^1)\otimes \mathcal{B}=\mathcal{B}\times \mathcal{B}$ is not a fibration. In fact let $\theta: x\overset{\sim}{\to} y$ be an isomorphism up to homotopy between two different objects in $\mathcal{B}$, then $(\theta,id):(x,x)\overset{\sim}{\to} (y,x)$ is an isomorphism up to homotopy in $\mathcal{B}\times \mathcal{B}$ but $(\theta,id)$ cannot be in the image of $p$ because all objects in the image of $p$ are of the form $(x,x)$.
\end{rmk}

The following result follows from the definitions:
\begin{prop}\label{prop: simplicial resolution and images of simplicial sets}
$\Delta_{n}(\mathcal{B})$ is a model for the mapping dg-category $\mathcal{B}^{\Delta[n]}$.
\end{prop}

\subsection{The totalization and homotopy limit of cosimplicial dg-categories}\label{subsection: homotopy limit in dgCat}
Let $\{\mathcal{B}^{\bullet}\}$ be a cosimplicial object in $\dgCat_{\DK}$. We first study its totalization $\Tot(\mathcal{B}^{\bullet})$. Definition \ref{defi: total object of cosimplicial object} tells us that $\Tot(\mathcal{B}^{\bullet})$ is the equalizer of
$$
\prod_{[n]\in Ob(\mathbf{\Delta})}\Delta_n(\mathcal{B}^n)\underset{\psi}{\overset{\phi}{\rightrightarrows}}\prod_{(\sigma:[n]\to [m])\in Mor(\mathbf{\Delta})}\Delta_n(\mathcal{B}^m).
$$
hence it is a subcategory of $\prod_{[n]\in Ob(\mathbf{\Delta})}\Delta_n(\mathcal{B}^n)$.
Let $\prod_n(\mathcal{E}^n,\phi^n)$ be an object in $\Tot(\mathcal{B}^{\bullet})$ where each $(\mathcal{E}^n,\phi^n)$ is an object in $\Delta_n(\mathcal{B}^n)$. Recall that $\mathcal{E}^n$ consists of a collection of object  $\{E^n_0,\ldots, E^n_n\}$ in $\mathcal{B}^n$. The basic idea of the construction is to get rid of the redundant data. First we have the following proposition.

\begin{prop}\label{prop: E in totalization}
Let $\prod_n(\mathcal{E}^n,\phi^n)$ be an object in $\Tot(\mathcal{B}^{\bullet})$. For each $n$ let $d^n_i:[0]\to [n]$ be the map that sends $0$ to $i$, then for each $n$ and $i$ we have
$$
E^n_i=d^n_i(E^0_0).
$$
In other words, each $\mathcal{E}^n$ is determined by $\mathcal{E}^0=\{E^0_0\}$.
\end{prop}
\begin{proof}
Recall that $\Tot(\mathcal{B}^{\bullet})$ is the equalizer of
$$
\prod_{[n]\in Ob(\mathbf{\Delta})}\Delta_n(\mathcal{B}^n)\underset{\psi}{\overset{\phi}{\rightrightarrows}}\prod_{(\sigma:[n]\to [m])\in Mor(\mathbf{\Delta})}\Delta_n(\mathcal{B}^m).
$$
Now pick $\sigma=d^n_i: [0]\to [n]$ we get
$$
\begin{tikzcd}[column sep=small,row sep=small]
\Delta_0(\mathcal{B}^0)\arrow{dr}{\Delta_0((d^n_i)_*)}& \\
 &\Delta_0(\mathcal{B}^n)\\
\Delta_n(\mathcal{B}^n)\arrow{ur}[swap]{(d^n_i)^*(\mathcal{B}^n)}&.
\end{tikzcd}
$$
Recall by Definition \ref{defi: simplicial maps between the n-simplexes in a dg-category} we have $(d^n_i)^*\mathcal{E}^n=\{E^n_i\}$. Therefore $\prod_n(\mathcal{E}^n,\phi^n)$ belongs to its equalizer implies that   $E^n_i=d^n_i(E^0_0)$.
\end{proof}

Now we move on to study the $\phi^n$'s in $\Tot(\mathcal{B}^{\bullet})$. First we need the following observation.

\begin{lemma}\label{lemma: exists a given order-preserving map}
Let $I=(i_0,\ldots, i_k)$, $0\leq i_0\leq \ldots \leq i_k\leq n$ be a multi-index. Then there exists a unique order-preserving map
$$
\sigma: [k]\to [n]
$$
such that $\sigma(j)=i_j$ for all $0\leq j\leq k$.
\end{lemma}
\begin{proof}
It is self-evident.
\end{proof}

\begin{defi}\label{defi: k-th standard morphism}
For any $k\geq 1$ we call the morphism $\phi^{k,1-k}_{01\ldots k}$ in $\mathcal{B}^k$ the \emph{$k$th standard morphism}.
\end{defi}

\begin{prop}\label{prop: phi in totalization}
Let $\prod_n(\mathcal{E}^n,\phi^n)$ be an object in $\Tot(\mathcal{B}^{\bullet})$. Let $I=(i_0,\ldots, i_k)$, $0\leq i_0\leq \ldots \leq i_k\leq n$ be a multi-index and $\phi^{k,1-k}_{i_0\ldots i_k}$ be the corresponding morphism in $\mathcal{B}^n$. Let $\phi^{k,1-k}_{01\ldots k}$  be the $k$th standard morphism in $\mathcal{B}^k$ and $\sigma: [k]\to [n]$ be the map as in Lemma \ref{lemma: exists a given order-preserving map}. Then we have
$$
\phi^{k,1-k}_{i_0\ldots i_k}=\sigma_*(\phi^{k,1-k}_{01\ldots k})\in \Hom^{1-k}_{\mathcal{B}^n}(d^n_{i_k}(E^0_0),d^n_{i_0}(E^0_0)).
$$
In other words, each morphism is determined by the  standard morphisms.
\end{prop}
\begin{proof}
The proof is very similar to that of Proposition \ref{prop: E in totalization} and is left as an exercise.
\end{proof}

From now on, when we talk about a morphism $\phi^{k,1-k}_{i_0\ldots i_k}$, we always think it as $\sigma_*(\phi^{k,1-k}_{01\ldots k})$.

\begin{thm}\label{thm: totalization in dgCat}
Let $\{\mathcal{B}^{\bullet}\}$ be a  cosimplicial object in $\dgCat_{\DK}$. An object of $\Tot(\mathcal{B}^{\bullet})$ consists of a pair $(E,\phi)$ where $E$ is an object in $\mathcal{B}^0$ and $\phi$ is a collection of standard morphisms
$$
\phi^{k,1-k}_{0\ldots k}\in \Hom^{1-k}_{\mathcal{B}^k}(d^k_k(E),d^k_0(E))
$$
which satisfy the following two conditions
\begin{enumerate}
\item the Maurer-Cartan equation
\begin{equation}
\D\phi+\phi\cdot \phi=0
\end{equation}
or more explicitly
\begin{equation}\label{equation: MC for totalization explicit}
(-1)^kd_{\mathcal{B}}(\phi^{k,1-k}_{0\ldots k})+\sum_{j=1}^{k-1}(-1)^j\phi^{k-1,2-k}_{0\ldots \widehat{j}\ldots k}+\sum_{j=1}^{k-1}(-1)^{(1-j)(k-j)}\phi^{j,1-j}_{0\ldots j}\circ \phi^{k-j,1-k+j}_{j\ldots k}=0.
\end{equation}

\item $\phi^{1,0}_{01}$ is invertible in the homotopy category $\text{Ho}(\mathcal{B}^1)$.
\end{enumerate}

A degree $m$ morphism $\theta$ from $(E, \phi)$ to $(F,\psi)$ consists of  a collection of morphism
$$
\theta^{k,m-k}_{0\ldots k}\in \Hom^{m-k}_{\mathcal{B}^k}(d^k_k(E),d^k_0(F)), \text{ for all } k\geq 0
$$
and the differential on $\theta$ is given by
\begin{equation}
d\theta=\D \theta+\psi\cdot \theta-(-1)^m\theta\cdot \phi.
\end{equation}
\end{thm}
\begin{proof}
It is consequence of Definition \ref{defi: total object of cosimplicial object}, Proposition \ref{prop: E in totalization} and Proposition \ref{prop: phi in totalization}.
\end{proof}

\begin{coro}\label{coro: homotopy limit in dgcat explicit formula}
Let $\{\mathcal{B}^{\bullet}\}$ be a \emph{Reedy fibrant}  cosimplicial object in $\dgCat_{\DK}$. Then the construction in Theorem \ref{thm: totalization in dgCat} gives an explicit formula of the homotopy limit of  $\{\mathcal{B}^{\bullet}\}$.
\end{coro}
\begin{proof}
It is a direct corollary of Theorem \ref{thm: totalization and homotopy limit are weakly equi for fibrant object} and Theorem \ref{thm: totalization in dgCat}.
\end{proof}

\section{Twisted complexes}\label{section: twisted complexes}
\subsection{Rectification}\label{subsection: rectification}
Let $\mathcal{S}$ be a class of ringed spaces and $\mathcal{F}$ be a functor from $\mathcal{S}^{op}$ to dgCat.

\begin{eg}
We can take $\mathcal{S}$ to be the category of schemes over a field $k$ with $\text{char}(k)=0$ and for any $k$-scheme $X$, $\mathcal{F}(X)$ is the dg-category of complexes of $\mathcal{O}_X$-modules on $X$.
\end{eg}

Now let $X_{\bullet}$ be a simplicial object in $\mathcal{S}$, then $\mathcal{F}(X_{\bullet})$ becomes a cosimplicial object in dgCat. Then we can apply Corollary \ref{coro: homotopy limit in dgcat explicit formula} to this case.

Now we have a problem: dgCat is a $2$-category (See \cite{tamarkin2007dg}) and in general $\mathcal{F}: \mathcal{S}^{op}\to \dgCat$  is only a pseudo-functor hence $\mathcal{F}(X_{\bullet})$ is a pseudo-cosimplicial object in dgCat. For example when $\mathcal{F}=Cpx$, for any morphisms $f: X\to Y$ and $g: Y\to Z$ between schemes,  we have $f^*: Cpx(Y)\to Cpx(X)$ and $g^*: Cpx(Z)\to Cpx(Y)$. However we do not have
$$
f^*g^*=(gf)^*
$$
because they are only canonically isomorphic. See \cite{vistoli2005grothendieck} Section 3.2.1.

We have this problem in Section \ref{subsection: group action}, and thus we need to find a canonical rectification of the diagram. The construction is as follows. First we   have the following definition inspired by \cite{vistoli2005grothendieck} Definition 3.10.

\begin{defi}\label{defi: dg-pseudo functor}
A \emph{dg pseudo-functor} $\mathcal{F}$ on $\mathcal{S}$ consists of the following data
\begin{enumerate}
\item For each object $X$ of $\mathcal{S}$ a dg-category $\mathcal{F}(X)$;
\item For each morphism $f: X\to Y$ a dg-functor $f^*: \mathcal{F}(Y)\to \mathcal{F}(X)$;
\item For each object $X$ of $\mathcal{S}$ a (strict) isomorphism $\epsilon_X:id_X^*\overset{\cong}{\to} id_{\mathcal{F}(X)}$ in the dg-category $\text{Fun}(\mathcal{F}(X), \mathcal{F}(X))$;
\item For each pair of arrows $X\overset{f}{\to}Y\overset{g}{\to}Z$ a (strict) isomorphism
$$
\alpha_{f,g}: f^*g^*\overset{\cong}{\to} (gf)^* \text{ in } \text{Fun}(\mathcal{F}(Z), \mathcal{F}(X)).
$$
\end{enumerate}
These data are required to satisfy the following conditions.
\begin{enumerate}[a.]
\item If  $f: X\to Y$ is a morphism in $\mathcal{S}$ and $\eta$ is an object in $\mathcal{F}(Y)$, then we have
$$
\alpha_{id_X,f}(\eta)=\epsilon_X(f^*\eta) \text{ in } \Hom_{\mathcal{F}(X)}(id_X^*f^*\eta,f^*\eta)
$$
and
$$
\alpha_{f,id_Y}(\eta)=f^*\epsilon_Y(\eta) \text{ in } \Hom_{\mathcal{F}(Y)}(f^*id_Y^*\eta,f^*\eta).
$$
\item For any morphisms $X\overset{f}{\to}Y\overset{g}{\to}Z\overset{h}{\to} W$ and an object $\theta$ in $\mathcal{F}(W)$, , the
diagram
$$
\begin{CD}
f^*g^*h^*\theta @>\alpha_{f,g}(h^*\theta)>> (gf)^*h^*\theta\\
@VVf^*\alpha_{g,h}(\theta)V @VV\alpha_{gf,h}(\theta)V\\
f^*(hg)^*\theta @>\alpha_{f,hg}(\theta)>> (hgf)^*\theta
\end{CD}
$$
strictly commutes.
\end{enumerate}
\end{defi}

\begin{rmk}
Notice that in Definition \ref{defi: dg-pseudo functor} we require that all isomorphisms are strict and all diagrams strictly commute, i.e. not just up to chain homotopy. One could consider more general $(\infty,1)$-functors from $\mathcal{S}$ to dgCat considered as an $(\infty,1)$-category. However we use this definition because on the one hand it simplifies the construction and on the other hand the pseudo-functors we are considering below, like $Cpx$ and $Perf$, fit into Definition \ref{defi: dg-pseudo functor}.
\end{rmk}

Now we will describe how to rectify this pseudo-functor to obtain a diagram. First we consider $\mathcal{F}$ as a dg-category fibered over $\mathcal{S}$. For a fixed $X$ in $\mathcal{S}$, the forgetful functor $\mathcal{S}/X\to \mathcal{S}$ makes $\mathcal{S}/X$  a fibered category over $\mathcal{S}$ and we can consider it as a fibered dg-category over $\mathcal{S}$. Then we define
\begin{equation}
\widetilde{\mathcal{F}}(X)=\text{Fun}_{\mathcal{S}}(\mathcal{S}/X,\mathcal{F})
\end{equation}
where the right hand side is the dg-category of dg-functors fibered over $\mathcal{S}$ between $\mathcal{S}/X$ and $\mathcal{F}$. In more details, a dg-functors $\xi: \mathcal{S}/X\to \mathcal{F}$ fibered over $\mathcal{S}$ sends each $h: Y\to X$ in $\mathcal{S}/X$ to an object $\xi(h)$ of the dg-category $\mathcal{F}(Y)$ and each morphism in $\mathcal{S}/X$
$$
\xymatrix{
  Y \ar[rr]^{g} \ar[dr]_{}
                &  &   Z \ar[dl]^{}    \\
                & X                 }
$$
to a degree $0$ closed morphism  $\xi_g:\xi(Y)\to f^*\xi(Z))$ in $\mathcal{F}(Y)$ which is further strictly invertible.

Moreover   any morphism $f: X\to Y$ gives a morphism of fibered categories
$$
\xymatrix{
  \mathcal{S}/X \ar[rr]^{\mathcal{S}/f} \ar[dr]_{}
                &  &  \mathcal{S}/Y \ar[dl]^{}    \\
                & \mathcal{S}                 }
$$
and hence the composition induces a dg-functor
$$
\widetilde{f^*}: \widetilde{\mathcal{F}}(Y)\to \widetilde{\mathcal{F}}(X).
$$

For $\widetilde{\mathcal{F}}$ we have the following lemma.

\begin{lemma}\label{lemma: F tilde is a strict functor}
$\widetilde{\mathcal{F}}: \mathcal{S}^{op}\to \dgCat$ is a strict functor.
\end{lemma}
\begin{proof}
It is obvious from the definition of $\widetilde{f^*}$.
\end{proof}

In general $ \widetilde{\mathcal{F}}(X)$ is a much larger dg-category than $\mathcal{F}(X)$. Nevertheless we have the following important result.

\begin{prop}[Dg-categorical Yoneda Lemma]\label{prop: F tilde is weakly equivalent to F}
For each $X$ there is a natural equivalence of categories realised by a dg-functor
$$
\Phi: \widetilde{\mathcal{F}}(X)\overset{\sim}{\to} \mathcal{F}(X).
$$
In fact $\Phi$ is essentially surjective on objects and degreewise fully faithful on morphisms. (In particular $\Phi$ is a DK equivalence.)
\end{prop}
\begin{proof}
The proof is essentially the same as that of \cite{vistoli2005grothendieck} 3.6.2 and we give it here for completeness.

For any object $\xi$ in $\widetilde{\mathcal{F}}(X)$ we define $\Phi(\xi)$ to be $\xi(id_X)\in \mathcal{F}(X)$. For each natural transformation $T: \xi\to \eta$ we define $\Phi(T)$ to be the morphism $T_{id_X}: \xi(id_X)\to \eta(id_X)$. It is clear that $\Phi: \widetilde{\mathcal{F}}(X)\overset{\sim}{\to} \mathcal{F}(X)$ so defined is a dg-functor.

Then we need to show that $\Phi$ is  essentially surjective and  degreewise fully faithful. For the surjectivity, consider an arbitrary object $E\in \mathcal{F}(X)$, we want to construct a $\mathcal{S}$-dg-functor $\xi: \mathcal{S}/X\to \mathcal{F}$ such that $\xi(id_X)=E$. The functor $\xi$ is defined as follows. Given an object $h:Y\to X$ in $\mathcal{S}/X$ we define $\xi(h)=h^*E$. For a morphism
$$
\xymatrix{
  Y \ar[rr]^{g} \ar[dr]_{h}
                &  &   Z \ar[dl]^{f}    \\
                & X                 }
$$
in  $\mathcal{S}/X$ we know that there is a natural isomorphism $\alpha_{f,g}: g^*f^*\overset{\sim}{\to} h^*$ hence for $E$ we get
$$
\alpha_{f,g}(E):  g^*f^*E\to h^*E.
$$
whose inverse $\alpha_{f,g}^{-1}(E): h^*E\to g^*f^*E$ gives $\xi(g): \xi(h)\to \xi(f)$. This proves the   surjectivity on objects. Notice that $\alpha_{f,g}$ is strictly invertible so we do not need to choose a homotopic inverse of it.

As for the fully-faithfulness, for any natural transformation $T: \xi\to \eta$ and any $f: Y\to X$ we consider the diagram
$$
\xymatrix{
  Y \ar[rr]^{f} \ar[dr]_{f}
                &  &   X \ar[dl]^{id_X}    \\
                & X                 }
$$
The functoriality of $T$ gives the following commutative diagram
$$
\begin{CD}
\xi(f) @>T_f>> \eta(f)\\
@VV\xi_fV @VV\eta_fV\\
f^*\xi(id_X) @>f^*T_{id_X}>> f^*\eta(id_X)
\end{CD}
$$
By definition we know that the vertical arrows $\xi_f$ and $\eta_f$ are strictly invertible hence for any $f: Y\to X$, $T_f$ is uniquely determined by $T_{id_X}$ through the formula
$$
T_f=(\eta_f)^{-1}f^*T_{id_X}\xi_f.
$$
This proves the fully-faithfulness.
\end{proof}

Therefore for a simplicial space $X_{\bullet}$ we can replace the associated non-strict cosimplicial diagram $\mathcal{F}(X_{\bullet})$ by a strict cosimplicial diagram $\widetilde{\mathcal{F}}(X_{\bullet})$. From now on when we talk about the totalization, homotopy limit or Reedy fibrancy of $\mathcal{F}(X_{\bullet})$, we always mean the corresponding construction or property of $\widetilde{\mathcal{F}}(X_{\bullet})$.

\begin{rmk}
In this paper, rectification will be used mainly in Section \ref{subsection: group action} below.
\end{rmk}

\subsection{A criterion of Reedy fibrancy}\label{subsection: Criterion of Reedy fibrancy}
Before moving on to examples, we want to have a criterion on the Reedy fibrancy of $\mathcal{F}(X_{\bullet})$. First recall the definition of split simplicial spaces.

\begin{defi}[\cite{dugger2004hypercovers} Definition 4.13]\label{defi: split simplicial spaces}
A simplicial space $X_{\bullet}$ is said to be \emph{split}, or to \emph{have free degeneracies}, if there exist subspaces $N_k\hookrightarrow X_k$ such that the canonical maps
$$
N_k\coprod L_kX \to X_k
$$
are isomorphisms for all $k \geq 0$, where $L_kX $ is the $k$-th latching object of $X_{\bullet}$. This is equivalent to requiring
that the canonical map
$$
\coprod_{\sigma}N_{\sigma}\to X_k
$$
is an isomorphism, where the variable $\sigma$ ranges over all surjective maps in $\bf{\Delta}$ of the form $[k]\to [n]$, $N_{\sigma}$ denotes a copy of $N_n$, and the map $N_{\sigma}\to X_k$ is the one induced
by $\sigma^*: X_n\to X_k$.
\end{defi}

The idea is that the objects $N_k$ represent the non-degenerate part of $X_{\bullet}$ in
dimension $k$, and that the leftover degenerate part is as free as possible.

\begin{prop}\label{prop: split and left exact implies Reedy fibrant}
Let $\mathcal{F}: S^{op}\to \text{dgCat}$ be a functor which sends finite coproducts to products. Then for any split simplicial object  $X_{\bullet}$ in $\mathcal{S}$, the cosimplicial object $\mathcal{F}(X_{\bullet})$ is Reedy fibrant.
\end{prop}
\begin{proof}
First we notice that if $\mathcal{F}: S^{op}\to \text{dgCat}$   sends finite coproducts to products, then so is $\widetilde{\mathcal{F}}$. Now it is sufficient to show that the natural map $\widetilde{\mathcal{F}}(X_k)\to M_k(\widetilde{\mathcal{F}}(X_{\bullet}))$ is a fibration in $\dgCat_{\DK}$ for any $k\geq 0$, where $M_k(\widetilde{\mathcal{F}}(X_{\bullet}))$ is the $k$th matching object of  $\widetilde{\mathcal{F}}(X_{\bullet})$. Recall by Tabuada's definition \cite{tabuada2005structure} of Dwyer-Kan model structure on dgCat, a functor $G: \mathcal{C} \to \mathcal{D}$ of dg-categories is a fibration if:
\begin{itemize}
\item[F1] $G$ induces surjections on Hom-spaces.
\item[F2] Given $u: G(c) \to d$ such that $[u]$ is an isomorphism in $Ho(\mathcal{D})$, there is $v: c \to c'$ such that $G(c')=d$ and $F(v) = u$.
\end{itemize}

We need the following lemma on the matching object.

\begin{lemma}\label{lemma: matching object in sheaves on split simplicial space}
Let $\widetilde{\mathcal{F}}: S^{op}\to \text{dgCat}$ be a functor which sends finite coproducts to products. Let $X_{\bullet}$ be a split simplicial object and $
\coprod_{\sigma}N_{\sigma}\overset{\cong}{\to} X_k
$ is given as in Definition \ref{defi: split simplicial spaces}. Then for any $k\geq 0$ the $k$th matching object $M_k(\widetilde{\mathcal{F}}(X_{\bullet}))$ is isomorphic to
$$
\prod_{\substack{\sigma: [k]\to [n]\text{ surjective}\\ n<k}}\widetilde{\mathcal{F}}(N_{\sigma}).
$$
\end{lemma}
\begin{proof}[Proof of Lemma \ref{lemma: matching object in sheaves on split simplicial space}]
Since $X_{\bullet}$ is split and $\widetilde{\mathcal{F}}$ sends finite coproducts to products, we get
$$
\widetilde{\mathcal{F}}(X_k)=\prod_{\sigma: [k]\to [n]\text{ surjective}}\widetilde{\mathcal{F}}(N_{\sigma}).
$$

Then the result of Lemma \ref{lemma: matching object in sheaves on split simplicial space} is clear from the construction of the matching object.
\end{proof}

By Lemma \ref{lemma: matching object in sheaves on split simplicial space} we know that $\widetilde{\mathcal{F}}(X_k)=\widetilde{\mathcal{F}}(N_k)\oplus M_k(\widetilde{\mathcal{F}}(X_{\bullet}))$. The natural map $\widetilde{\mathcal{F}}(X_k)\to M_k(\widetilde{\mathcal{F}}(X_{\bullet}))$ is exactly the projection onto the second summand hence by definition it is a fibration. We finishes the proof of Proposition \ref{prop: split and left exact implies Reedy fibrant}.
\end{proof}

\subsection{Twisted complexes}\label{subsection: twisted complexes}
In this subsection by ringed space we mean a topological space $X$ together with a sheaf of (not necessarily commutative) ring $\mathcal{R}$ on $X$. Examples include
\begin{itemize}
\item A scheme $X$ with the structure sheaf $\mathcal{O}_X$;
\item A complex manifold $X$ with the  sheaf of analytic functions $\mathcal{O}_X$;
\item A topological space $X$ with the constant sheaf of rings $\underline{\mathbb{C}}$;
\item A scheme $X$ with the sheaf of  rings of differential operators $\mathcal{D}_X$.
\end{itemize}
Let
$$Cpx: \text{Space}^{op}\rightarrow \dgCat$$
be the contravariant functor which assigns to each  ringed space $X$ the dg-category of complexes of left $\mathcal{R}$-modules on $X$. This is a presheaf of dg-categories.

\begin{rmk}
In practice we are often less interested in the category of all complexes $\mathcal{R}$-modules than in some well-behaved subcategory, say complexes with quasi-coherent cohomology on a scheme, or $\mathcal D_{X}$-modules which are quasi-coherent as $\mathcal O_{X}$-modules. As long as the condition we impose is local the theory works equally well in those cases. We will explicitly consider the case of perfect complexes in the next section.
\end{rmk}

\begin{rmk}
In this subsection by $\mathcal{R}$-modules we always mean left $\mathcal{R}$-modules, unless it is explicitly pointed out otherwise.
\end{rmk}

Let $(X,\mathcal{R})$ be a  ringed space and $\mathcal{U}=\{U_i,~ i\in I\}$ be a locally finite open cover of $X$. We consider the \emph{\v{C}ech nerve} of the covering $\mathcal{U}$, $N(\mathcal{U})$, which is a simplicial space
$$
\begin{tikzcd}
\coprod U_i & \coprod U_i\cap U_j  \arrow[yshift=0.7ex]{l}\arrow[yshift=-0.7ex]{l} &  \coprod U_i\cap U_j\cap U_k  \arrow[yshift=1ex]{l}\arrow{l}\arrow[yshift=-1ex]{l}\cdots
\end{tikzcd}
$$
For simplicity let $U_{i_0\ldots i_n}$ denote the intersection $U_{i_0}\bigcap\ldots \bigcap U_{i_n}$. Moreover we use the notation $N_{\bullet}$ to denote the \v{C}ech nerve, i.e. we have
$$
N_k=\coprod_{(i_0\ldots i_k)} U_{i_0\ldots i_k}  \text{ for any } k\geq 0.
$$
Therefore
$$
Cpx(N_k)=\prod_{(i_0\ldots i_k)} Cpx(U_{i_0\ldots i_k})  \text{ for any } k\geq 0.
$$

We get a cosimplicial diagram of dg-categories
\begin{equation}\label{equation: second cosimplicial diagram of Cpx of an open cover}
\begin{tikzcd}
Cpx(N_0) \arrow[yshift=0.7ex]{r}\arrow[yshift=-0.7ex]{r}& Cpx(N_1) \arrow[yshift=1ex]{r}\arrow{r}\arrow[yshift=-1ex]{r}  &   Cpx(N_2) \cdots
\end{tikzcd}
\end{equation}

\begin{rmk}\label{rmk: open cover gives a strict cosimplicial diagram}
Note that Diagram \ref{equation: second cosimplicial diagram of Cpx of an open cover} is a strict cosimplicial diagram because the coface maps are given by restrictions to open subsets. Therefore we do not need to rectify the diagram.
\end{rmk}

We want to achieve an explicit description of the homotopy limit of Diagram \ref{equation: second cosimplicial diagram of Cpx of an open cover}. For this we introduce \emph{twisted complexes}.

Let us first set up some notations. For each $U_{i_k}$, let $E^{\bullet}_{i_k}$ be a   graded sheaf of $\mathcal{R}$-modules on $U_{i_k}$. Let
\begin{equation}\label{equation: bigraded sheaves}
C^{\bullet}(\mathcal{U},E^{\bullet})=\prod_{p,q}C^p(\mathcal{U},E^q)
\end{equation}
be the bigraded complexes of $E^{\bullet}$. More precisely, an element $c^{p,q}$ of $C^p(\mathcal{U},E^q)$ consists of a section $c^{p,q}_{i_0\ldots i_p}$ of $E^{q}_{i_0}$ over each non-empty intersection $U_{i_0\ldots i_p}$. If $U_{i_0\ldots i_p}=\emptyset$, simply let the component on it be zero.

Now if another   graded sheaf $F^{\bullet}_{i_k}$ of $\mathcal{R}$-modules is given on each $U_{i_k}$, then we can consider the map
\begin{equation}\label{equation: map with bigrade between graded sheaves}
C^{\bullet}(\mathcal{U},\text{Hom}^{\bullet}(E,F))=\prod_{p,q}C^p(\mathcal{U},\text{Hom}^q(E,F)).
\end{equation}
An element $u^{p,q}$ of $C^p(\mathcal{U},\text{Hom}^q(E,F))$ gives a section $u^{p,q}_{i_0\ldots i_p}$ of $\text{Hom}^q_{\mathcal{R}-\text{Mod}}(E^{\bullet}_{i_p},F^{\bullet}_{i_0})$, i.e. a degree $q$ map from $E^{\bullet}_{i_p}$ to $F^{\bullet}_{i_0}$ over the non-empty intersection $U_{i_0\ldots i_p}$. Notice that we require $u^{p,q}$ to be a map from the $E$ on the last subscript of $U_{i_0\ldots i_p}$ to the $F$ on the first subscript of $U_{i_0\ldots i_p}$. Again, if $U_{i_0\ldots i_p}=\emptyset$,  let the component on it be zero.

The compositions of $C^{\bullet}(\mathcal{U},\text{Hom}^{\bullet}(E,F))$ is the shuffled multiplication given in Section \ref{subsection: simplicial resolution in dgCat} Equation (\ref{equation: shuffled multiplication in simplicial dg}).
In particular $C^{\bullet}(\mathcal{U},\text{Hom}^{\bullet}(E,E))$ becomes an associative algebra under this composition (It is easy but tedious to check the associativity).

There are also \v{C}ech-style differential operator $\delta$ on $C^{\bullet}(\mathcal{U},\text{Hom}^{\bullet}(E,F))$  of bidegree $(1,0)$ given by the following formula which is similar to Equation (\ref{equation: D in simplicial dg}).
\begin{equation}\label{equation: delta on maps}
(\delta u)^{p+1,q}_{i_0\ldots i_{p+1}}=\sum_{k=1}^p(-1)^k u^{p,q}_{i_0\ldots \widehat{i_k} \ldots i_{p+1}}|_{U_{i_0\ldots i_{p+1}}} \,\text{ for } u^{p,q}\in C^p(\mathcal{U},\text{Hom}^q(E,F)).
\end{equation}

\begin{ctn}
Notice that the map $\delta$ defined above is different from the usual \u{C}ech differential since we do not include the $0$th and the $p+1$th indices.
\end{ctn}

\begin{prop}\label{prop: Leibniz for Cech differential}
The  differential satisfies the Leibniz rule. More precisely we have
$$
\delta(u\cdot v)=(\delta u)\cdot v+(-1)^{|u|}u\cdot (\delta v)
$$
and
$$
\delta(u\cdot c)=(\delta u)\cdot c+(-1)^{|u|}u\cdot (\delta c)
$$where $|u|$ is  the total degree of $u$.
\end{prop}
\begin{proof}
This is a routine check.
\end{proof}

Now we can define twisted complexes.

\begin{defi}[\cite{o1981trace} Definition 1.3 or \cite{ZhWei20151} Definition 5]\label{defi: twisted complex}
Let $(X,\mathcal{R})$ be a ringed  space and $\mathcal{U}=\{U_i\}$ be a locally finite open cover of $X$. A \emph{twisted complex} consists of a  graded sheaves $E^{\bullet}_{i}$ of  $\mathcal{R}$-modules on each $U_i$ together with
$$
a=\sum_{k \geq 0} a^{k,1-k}
$$
where $a^{k,1-k}\in C^k(\mathcal{U},\text{Hom}^{1-k}(E,E)) $ such that they satisfy the Maurer-Cartan equation
\begin{equation}\label{equation: MC for twisted complex}
\delta a+ a\cdot a=0.
\end{equation}
More explicitly, for $k\geq 0$ and all multi-index $(i_0\ldots i_k)$ we have
\begin{equation}\label{equation: MC for twisted complex explicit}
\sum_{j=1}^{k-1}(-1)^ja^{k-1,2-k}_{i_0\ldots\widehat{ i_j}\ldots i_k}+ \sum_{l=0}^k (-1)^{(1-l)(k-l)}a^{l,1-l}_{i_0\ldots i_l} a^{k-l,1-k+l}_{i_l\ldots i_k}=0.
\end{equation}

Moreover we impose the following non-degenerate condition: for each $i$, the chain map
\begin{equation}\label{equation: nondegenerate condition of aii}
a^{1,0}_{ii}: (E^{\bullet}_i,  a^{0,1}_i)\to (E^{\bullet}_i,  a^{0,1}_i) \text { is invertible up to homotopy.}
\end{equation}

The twisted complexes on $(X,\mathcal{R}, \{U_i\})$ form a dg-category: the objects are the twisted complexes $\mathcal{E}=(E^{\bullet}_i,a)$ and the morphism from $\mathcal{E}=(E^{\bullet}_i,a)$ to $\mathcal{F}=(F^{\bullet}_i,b)$ are $C^{\bullet}(\mathcal{U},\text{Hom}^{\bullet}(E,F))$ with the total degree. Moreover, the differential on a morphism $\phi$ is given by
\begin{equation}\label{equation: differential on morphisms of twisted complexes}
d \phi=\delta \phi+b\cdot \phi-(-1)^{|\phi|}\phi\cdot a.
\end{equation}

We denote the dg-category of twisted complexes on $(X,\mathcal{R}, \{U_i\})$ by Tw$(X,\mathcal{R}, \{U_i\})$. If there is no risk of confusion we can simply denote it by Tw$(X)$.
\end{defi}

For more details on twisted complexes see \cite{ZhWei20151}. In this paper we want to find the relation between twisted complexes and $Cpx(N_{\bullet})$  where $N_{\bullet}$ is the \u{C}ech nerve of $\mathcal{U}$. Our first result is the following proposition.

\begin{prop}\label{prop: twisted complexes are totalization}
Let $(X,\mathcal{R})$ be a   ringed   space and $\mathcal{U}=\{U_i\}$ be a locally finite open cover of $X$. Recall that we use $N_{\bullet}$ to denote the classifying space of $\{U_i\}$. Then the dg-category of twisted complexes Tw$(X,\mathcal{R}, \{U_i\})$ is exactly the totalization of the cosimplicial object $Cpx(N_{\bullet})$ in $\dgCat_{\DK}$.
\end{prop}
\begin{proof}
First notice that Remark \ref{rmk: open cover gives a strict cosimplicial diagram} tells us we do not need to rectify the diagram.

We need to compare Theorem \ref{thm: totalization in dgCat} and Definition \ref{defi: twisted complex}. The apparent difference between these two is that in Theorem \ref{thm: totalization in dgCat} we only take $(0\ldots k)$ as the multi-index while in Definition \ref{defi: twisted complex} we take all multi-indices. The problem is solved once we notice that
$$
N_k=\coprod_{(i_0\ldots i_k)} U_{i_0\ldots i_k}
$$
and therefore   for all $k\geq 1$ the $\phi^{k,1-k}_{0\ldots k}$ gives an $a^{k,1-k}_{i_0\ldots i_k}$ on each $U_{i_0\ldots i_k}$. For $k=0$ we do not have $\phi^{0,1}_0$ but we know that
$$
a^{0,1}_{i_0}\cdot a^{k,1-k}_{i_0\ldots i_k}+ a^{k,1-k}_{i_0\ldots i_k}\cdot a^{0,1}_{i_k}=(-1)^k d(a^{k,1-k}_{i_0\ldots i_k}).
$$
As a result we get the correspondence between Equation (\ref{equation: MC for totalization explicit}) and Equation (\ref{equation: MC for twisted complex explicit}).

As for the non-degenerate condition, by the $k=2$ case of the Maurer-Cartan equation we have
$$
-a^{1,0}_{ii}+a^{1,0}_{ij}\cdot a^{1,0}_{ji}+a^{0,1}_i\cdot a^{2,-1}_{iji}+a^{2,-1}_{iji}\cdot a^{0,1}_i=0.
$$
In other words $a^{1,0}_{ii}$ is homotopic to $a^{1,0}_{ij}\cdot a^{1,0}_{ji}$ hence $a^{1,0}_{ii}$ is invertible in the homotopy category for each $i$ implies that $a^{1,0}_{ij}$ is also invertible in the homotopy category for each $i$, $j$. Therefore the two version of non-degenerate conditions are equivalent.

The proof of the equivalence between the morphism is the same.
\end{proof}

Next we want to study the relation between the twisted complexes and the homotopy limit. For this purpose we need to determine whether $Cpx(N_{\bullet})$ is Reedy fibrant. In fact the following proposition gives an affirmative answer.

\begin{prop}\label{prop: Cpx of open cover is Reedy fibrant}
Let $(X,\mathcal{R})$ be a ringed   space and $\mathcal{U}=\{U_i\}$ be a locally finite open cover of $X$. Let $N_{\bullet}$ be the classifying space of $\{U_i\}$ as before. Then
$Cpx(N_{\bullet})$ is a Reedy  fibrant cosimplicial object in $\dgCat_{\DK}$.
\end{prop}
\begin{proof}
It is easy to show that the \v{C}ech nerve of an open cover is a split simplicial space. It is also clear that the functor $Cpx$ sends finite coproducts to products hence the claim of Proposition \ref{prop: Cpx of open cover is Reedy fibrant} is a direct corollary of   Proposition \ref{prop: split and left exact implies Reedy fibrant}.
\end{proof}

\begin{coro}\label{coro: twisted complex is homotopy limit}
Let $(X,\mathcal{R})$ be a  ringed  space and $\mathcal{U}=\{U_i\}$ be a locally finite open cover of $X$. Let $N_{\bullet}$ be the classifying space of $\{U_i\}$ as before. Then
the dg-category of twisted complexes Tw$(X,\mathcal{R}, \{U_i\})$ gives an explicit construction of  $\holim Cpx(N_{\bullet})$.
\end{coro}
\begin{proof}
It is a direct corollary of Corollary \ref{coro: homotopy limit in dgcat explicit formula} and Proposition \ref{prop: Cpx of open cover is Reedy fibrant}.
\end{proof}

It is clear that Corollary \ref{coro: twisted complex is homotopy limit} applies to the following cases
\begin{itemize}
\item For a scheme $(X,\mathcal{O}_X)$, Tw$(X,\mathcal{O}_X, \{U_i\})$ is quasi-equivalent to the homotopy limit of sheaves of $\mathcal O$-modules on the nerve; \item Similarly for a complex manifold $(X,\mathcal{O}_X)$, Tw$(X,\mathcal{O}_X, \{U_i\})$ gives $\holim Cpx_{\mathcal{O}\text{-mod}}(N_{\bullet})$;
\item For a scheme $X$ with the sheaf of rings of differential operators $\mathcal{D}_X$ we obtain a category of twisted complexes of $ \mathcal{D}_{X}$-modules, Tw$(X,\mathcal{D}_X, \{U_i\})$, which is quasi-equivalent to $\holim Cpx_{\mathcal{D}\text{-mod}}(N_{\bullet})$
\item For a topological space $X$ with the constant sheaf of rings $\underline{\mathbb{C}}$ we obtain $Tw(X, \mathbb C, \{U_{i}\})$  as a model for $\holim_{\mathbb{C}\text{-mod}}(N_{\bullet})$. By inspection this is equivalent to the dg-category of $\infty$-local systems on $N_{\bullet}$, see \cite{holstein2014properness}.
If we now assume that the cover $\mathcal U$ is good, i.e. all opens and their intersections are contractible, then this nerve is homotopy equivalent to singular simplices on $X$, and $\infty$-local systems on the nerve are just sheaves of $\mathbb C$-modules on $X$ that are homotopy locally constant, see \cite{holstein2015moritacohomology} for more on this situation.
\end{itemize}

\begin{rmk}
Note that twisted complexes of $\mathcal{D}$-modules are not to be confused with twisted $\mathcal{D}$-modules on $X$, which are completely different objects.
\end{rmk}

\begin{rmk}
We can also extend this to Grothendieck topologies. For example, let $X$ by a variety (or scheme, or algebraic stack) equipped with a locally finite \'etale cover $\cat U$. We consider a sheaf of rings $\cat R_{X}$ on the \'etale site, for example the constant sheaf $\mathbf Z/\ell$, and write $Cpx$ for the category of $\cat R_{X}$-modules. We may then define $Tw(X, \cat R_{X}, \cat U)$.
 After rectifying the diagram $Cpx(N_{\bullet})$ we can show the analogue of Proposition \ref{prop: twisted complexes are totalization}.
 It is clear that the nerve of $\cat U$ is split, thus an analogue of Corollary \ref{coro: twisted complex is homotopy limit} also holds and $Tw(X, \cat R_{X}, \cat U) \simeq \holim Cpx(\cat N_{\bullet})$.

A potentially interesting object of study would be to take a colimit over all \'etale hypercovers of $X$, as one does when studying the \'etale homotopy type. Note that we consider colimits of hypercovers in Section \ref{subsection: stack of twisted complexes}, but we will not return to this example.
\end{rmk}
\subsection{Twisted perfect complexes}\label{subsection: twisted perfect complexes}

We are often not interested in all $\mathcal{R}$-modules but only some more convenient subcategory.
If the sheaf of ring $\mathcal{R}$ is commutative, we can also consider the contravariant functor
$$
Perf: \text{Space}^{op}\rightarrow \textsl{dg-Cat}
$$
which assigns to each ringed space $X$  the dg-category of strictly perfect complexes of $\mathcal{R}$-modules on $X$. As before let $\mathcal{U}=\{U_i\}$  locally finite open cover of $X$ and $N_k=\coprod U_{i_0\ldots i_k}$ be the nerve of the cover.  Similar to Diagram \ref{equation: cosimplicial diagram of Cpx of an open cover} we have a cosimplicial diagram of dg-categories.

\begin{equation}\label{equation: cosimplicial diagram of perfect Cpx of an open cover}
\begin{tikzcd}
Perf(N_0) \arrow[yshift=0.7ex]{r}\arrow[yshift=-0.7ex]{r}&  Perf(N_1) \arrow[yshift=1ex]{r}\arrow{r}\arrow[yshift=-1ex]{r}  &   Perf(N_2) \cdots
\end{tikzcd}
\end{equation}

We also have the following variant of twisted complexes.
\begin{defi}\label{defi: twisted perfect complex}
A \emph{twisted perfect complex} $\mathcal{E}=(E^{\bullet}_i,a)$ is the same as twisted complex in Definition \ref{defi: twisted complex} except that each $E^{\bullet}_i$ is a strictly perfect complex on $U_i$.

The twisted perfect complexes also form a dg-category and we denote it by $\text{Tw}_{\text{perf}}(X,\mathcal{R}, \{U_i\})$ or simply $\text{Tw}_{\text{perf}}(X)$. Obviously $\text{Tw}_{\text{perf}}(X)$ is a full dg-subcategory of Tw$(X)$.
\end{defi}

We have the following result for twisted perfect complexes which is similar to Proposition \ref{prop: twisted complexes are totalization} and Corollary \ref{coro: twisted complex is homotopy limit}.
\begin{prop}\label{prop: twisted perfect complexes are totalization}
Let $(X,\mathcal{R})$ be a  ringed  space and $\mathcal{U}=\{U_i\}$ be a locally finite open cover of $X$.  Then the dg-category of twisted complexes $\text{Tw}_{\text{perf}}(X,\mathcal{R}, \{U_i\})$ is exactly the totalization of the cosimplicial object $Perf(N_{\bullet})$ in $\dgCat_{\DK}$ and further weakly equivalent to the homotopy limit of $Perf(N_{\bullet})$.
\end{prop}
\begin{proof}
Since $Perf$ also send finite coproducts to products, the proof is exactly the same as that of Proposition \ref{prop: twisted complexes are totalization} and Corollary \ref{coro: twisted complex is homotopy limit}.
\end{proof}

The significance of twisted perfect complexes in geometry is given by the construction in \cite{o1981hirzebruch}. Moreover, we have the following result:

\begin{thm}\label{thm: twisted complexes enhancement}[\cite{ZhWei20151} Theorem 4]
Let $X$ be a quasi-compact and separated
or Noetherian scheme and $\{U_i\}$ be an affine cover, then $\text{Tw}_{\text{perf}}(X)$ gives a dg-enhancement of $D_{\text{perf}}(X)$, the derived category of perfect complexes on $X$.
\end{thm}
\begin{proof} See \cite{ZhWei20151} Theorem 4.
\end{proof}

\begin{rmk}
A systematic study of presheaves valued in sSet has been given in \cite{dugger2004hypercovers} and in \cite{holstein2015moritacohomology} Section 3.1 some results on the presheaves valued in dgCat has been achieved. We expect that the result in this section can help us further study the presheaves in dgCat.
\end{rmk}

\subsection{The stack of twisted perfect complexes}\label{subsection: stack of twisted complexes}

In this subsection we assume that $(X, \cat R)$ is a hereditarily paracompact ringed space
with a basis $\cat P$ of $p$-good open sets.

Recall from Section 3.3.2 of \cite{ZhWei20151} the following notation:
\begin{defi}
A locally ringed space $(U, \cat O_{U})$ is \emph{$p$-good} if any perfect complex of quasi-coherent sheaves on $U$ is quasi-isomorphic to a strictly perfect complex and quasi-coherent sheaves on $U$ have no higher cohomology.
A \emph{$p$-good cover} is a cover such that its nerve consists of a coproduct of $p$-good spaces in every degree.
\end{defi}

Interesting examples of $p$-good spaces are affine algebraic varieties, Stein spaces and contractible topological spaces whose structure sheaf is constant.

Thus our assumptions are satisfied by the following examples:
\begin{itemize}
\item separated schemes with $\cat P$ given by affine subschemes,
\item complex manifolds with $\cat P$ given by Stein subspaces,
\item locally contractible hereditarily paracompact topological spaces with constant structure sheaf, where $\cat P$ is any choice of a contractible basis.
\end{itemize}
\begin{rmk}
Note that as contractible open sets are not closed under intersection we have to make some choice in the last class of examples.
\end{rmk}
Now let $\cat C$ be the  category  given by a class of ringed spaces satisfying our assumptions.
Open covers of topological spaces form the basis of a Grothendieck topology $\tau$ on $\cat C$, giving us a site $(\cat C, \tau)$.

We note that $(\mathcal C, \tau)$ is a Verdier site. A Verdier site is just a Grothendieck site with a basis such that if $U \to X$ is a covering family then so is $U \to U\times_{X}U$. See Section 8 of \cite{dugger2004hypercovers} for details.
(Another example of a Verdier site is given by the \'etale topology on schemes.)

Let $y: \mathcal C \to PSh(\mathcal C, \tau)$ be the Yoneda embedding.

\begin{defi}
A \emph{hypercover} $\cat U = \{U_{\bullet}\}$ of an object $X \in \cat C$ is an augmented simplicial presheaf $U_{\bullet} \to y_{X}$ on $(\cat C, \tau)$ such that all $U_{n}$ are coproducts of representables and all $U_{n} \to M_{n}U$ are in $\tau$. (Here matching objects are computed over $y_X$, in particular $M_{0}U = y_{X}$.)
\end{defi}
 \begin{rmk}
 In \cite{dugger2004hypercovers} these are called basal hypercovers, but Theorem 8.6 in loc.\ cit.\ shows that any hypercover has a basal refinement. We explain below that this means it is enough to work with basal hypercovers.
  \end{rmk}

The canonical example of a hypercover is the nerve of a \v Cech covering.

 We say a hypercover is \emph{locally finite} if every cover $U_{n} \to M_{n}U$ is locally finite.
 A hypercover is \emph{split} if it satisfies Definition \ref{defi: split simplicial spaces} (replacing spaces by presheaves).
Given a collection $\cat P$ of open sets a hypercover is a \emph{$\mathcal P$-hypercover} if in every degree it consists of a coproduct of (sheaves represented by) objects in $\cat P$.

 The morphisms between hypercovers are \emph{refinements}, where we say $\cat U \to X$ is a refinement of $\cat V \to X$ if it factors through $\cat V \to X$, i.e.\ there are compatible factorizations $U_{n} \to V_{n}$ for every $n$. We call a refinement by a $\mathcal P$-hypercover a \emph{$\cat P$-refinement}.

 Now we note that we can define twisted complexes on a (locally finite) hypercover just as well as on a cover. Moreover, Proposition \ref{prop: Cpx of open cover is Reedy fibrant} and Proposition \ref{prop: twisted perfect complexes are totalization} still hold if $N_{\bullet}$ is a split hypercover. In the proof we only needed that $N_{\bullet}$ is split, not that it is the nerve of a cover.

  There is a contravariant functor of twisted complexes from the category of all locally finite hypercovers $\mathcal U$ of $X$ to $\dgCat$ that sends a refinement to the functor induced by restriction.

The class of all hypercovers is too large and we will work with the following smaller category of hypercovers:
\begin{defi}
Let $HC(X)$ denote the subcategory of locally finite, split $\mathcal P$-hypercovers of $X$.
We will sometimes write $HC$ if $X$ is clear from context.
  \end{defi}
  \begin{rmk}
  In the example of topological spaces with constant structure sheaf the choice of $\cat P$ is not canonical, we need to choose $\cat P$ for every $X$. But there are no compatibility conditions, and different choices of $\cat P$ do not affect the results.
 \end{rmk}
  Twisted complexes as we have defined them depend on a choice of hypercover. We now need to turn them into a functor on $\cat C$.
  Let $\cat C_{V}$ be the category whose objects are given by pairs $(X, \cat V)$ where $X \in \cat C$ and $\cat V \in HC(X)$,
  with morphisms given by maps of ringed spaces that are compatible with the hypercovers.

  There is a natural forgetful functor $u: \cat C_{V}\to \cat C$ which is full.
  \begin{defi}
  We denote by $Tw_{perf}: \cat C^{op} \to \dgCat$ the homotopy left Kan extension along $u$ of the natural functor $Tw_{perf}: \cat C_{V}^{op} \to \dgCat$ given by $(X, \mathcal V) \mapsto Tw_{perf}(X, \cat R_{X}, \cat V)$ on objects and by restriction on morphisms.
  \end{defi}
  Recall that the \emph{left Kan extension}, or the \emph{relative colimit}, of a functor $T: A \to  D$ along a functor $p: A \to B$ is defined by applying the left adjoint $p_{!}$ (if it exists) of the restriction functor $p^{*}: Fun(B,D) \to Fun(A,D)$ to $T$. The special case that $B$ is the 1-object category recovers the colimit of $T$. If $D$ is a model category then there is the corresponding notion of the homotopy left Kan extension, obtained as the left derived functor of $p_{!}$. See for example Section 10 of \cite{Dugger} for more details.

    \begin{rmk}
    Note that to construct the left Kan extension we need our index categories to be small. We may either restrict attention to a small subcategory of $\cat C$, or avoid this condition by assuming the existence of a suitable Grothendieck universe.
    \end{rmk}
     Our new definition of twisted complexes looks rather unwieldy, but we will show that we can still compute twisted complexes on $X$ using any hypercover in $HC(X)$.

  First we recall two very useful lemmas about computing homotopy colimits.
    Given a functor $\iota: I \to J$, recall the natural map $e_j: (j \downarrow \iota) \to J$ from the undercategory, sending $(i, j \rightarrow \iota(i))$ to $\iota(i)$.
\begin{lemma}\label{lemma-duggersthm-2}
Let $\iota: I \to J$ be a
functor between small categories such that for every $j \in J$ the undercategory $(j \downarrow \iota)$ is nonempty with a contractible nerve (we say $\iota$ is homotopy terminal) and let $X: J \to \cat M$ be a diagram with values in a model category. Then the map $\hocolim_{I} \iota^{*} X \to \hocolim_{J} X$ is a weak equivalence.
\end{lemma}
\begin{lemma}\label{lemma-duggerstheorem}
Let $\iota: I \to J$ be a functor between small categories and let $X: J \to \cat M$ a diagram with values in a model category. Suppose that the composition
\[\hocolim_{(j \downarrow \iota)} e_j^*(X) \to \colim_{(j \downarrow \iota)} e_j^*(X) \to X_{j}
\]
is a weak equivalence for every $j$. Then the natural map $\hocolim_I \iota^{*}X \to \hocolim_J X$ is a weak equivalence.
\end{lemma}
\begin{proof} [Proofs] For topological spaces these are Theorems 6.7 and 6.9 of \cite{Dugger} and the proofs (in Section 9.6) do not depend on the choice of model category.
\end{proof}

 \begin{lemma}\label{lemma-simplifykan}
 $Tw_{perf}(X,\mathcal R) \simeq \hocolim_{\mathcal U \in HC(X)^{op}} Tw_{perf}(X, \mathcal R, \mathcal U)$ where we take the homotopy colimit over the opposite of the category of locally finite, split $\mathcal P$-hypercovers $\mathcal U$ of $X$.
  \end{lemma}
  \begin{proof}\label{lemma-pointwisekan}
  We first use the pointwise computation of a homotopy left Kan extensions to deduce that $Tw_{perf}(X, \cat R) = \hocolim_{\cat C_{V}^{op}/X}Tw_{perf}(Y, \cat R_{Y}, \cat V)$. This is a special case of base change for relative colimits.
  See Proposition 10.2 of \cite{Dugger}, or Theorem 9.6.5 in \cite{RadulescuBanu06} for base change in the more general setting of relative categories.

  To simplify this computation we then use \ref{lemma-duggersthm-2} to show that the homotopy colimit can also be computed over $HC(X)^{op}$. To do this we need to show that the inclusion $\iota: HC(X)^{op} \to \cat C_{V}^{op}/X$ is homotopy terminal. For every $(Y, \cat V) \in \cat C_{V}$ with $f: X \to Y$ we need to show that $\iota \downarrow (Y, \cat V)$ is contractible.

We note that $f^{-1}(\cat V)$ is a hypercover of $X$, and it is clear that working over $(Y, \cat V)$ is equivalent to working over $f^{-1}(\cat V)$.

Contractibility of the overcategory follows if we can find an object $\cat U \in HC(X)$ such that the overcategory has all products with $\cat U$.
First assume we have a well-defined endofunctor $\cat W \to \cat W \times_{f^{-1}(\cat V)} \cat U$. Then the projection maps give natural transformations to the identity functor and projection to $\cat U$. On taking the nerve this gives a homotopy equivalence from the nerve of $\iota \downarrow f^{-1}(\cat V)$ to a point.

We will now construct a (locally finite, split) $\cat P$-refinement $\cat U$ of $f^{-1}(\cat V)$ that behaves like a hypercover of $f^{-1}(\cat V)$. We will now abuse notation and write $f^{-1}(\cat V)$ as $\cat V$ for ease of notation.
To be precise, we will inductively construct $\cat U$ as a hypercover over $\cat V$ that moreover satisfies the condition that $U_{0} \to V_{0}$ and all
$U_{n} \to M_{n}U\times_{M_{n}V} V_{n}$ are coverings.
For the base case we write $V_{0} = \amalg_{i}V_{0,i}$ and choose a locally finite covering of each $V_{0,i}$ using paracompactness. $U_{0}$ is their disjoint union. To check this covering is locally finite take $x \in X$, then $f(x) \in Y$ has a neighbourhood $W$ only intersecting finitely many $V_{0,i}$, and $f^{-1}(W)$ intersects only finitely many summands of $\cat U_{0}$.
For the induction step we note that $M_{n}U\times_{M_{n}V} V_{n}$ is again a coproduct of representables (this is true for $M_{n}$ by Lemma 8.4 in \cite{dugger2004hypercovers}, and then it easily follows for the fiber product). So we can find a locally finite $\cat P$-cover $U^{*}_{n} \to M_{n} U\times_{M_{n} V}V$ by the same argument as above, and let $U_{n} = U^{*}_{n} \amalg L_{n}U$. To see that this map is basal see the proof of Theorem 8.6 in \cite{dugger2004hypercovers}. It is split by construction.

 Given $\cat U \to \cat V$ as above and any refinement $\cat W \in HC(X)$ of $\cat V$ we can construct the levelwise pullback. It clearly satisfies the universal property. Now Lemma 7.1 in \cite{Stacks13} shows that a fiber product $\cat U\times_{\cat V}\cat W$ of hypercovers is a hypercover if $U_{0} \to V_{0}$ and
$U_{n} \to M_{n}U\times_{M_{n}V} V_{n}$ are coverings, which is true in our case by construction of $\cat U$. It remains to check that $\cat U\times_{\cat V}\cat W$ is a $\cat P$-hypercover, locally finite and split. It is locally finite  as $\cat U$ and $\cat W$ are. Similarly, it is a disjoint union of objects in $\cat P$ since the levelwise pullback is just a disjoint union of intersections, and $\cat P$ is closed under intersections. Finally, $\cat U\times_{\cat V}\cat W$ is split.
 We use that a hypercover $\mathcal U$ is split if $U_{n} = \amalg_{[n]\to [k]} N_{k}$ where the index goes over all surjections, see Definition \ref{defi: split simplicial spaces}.  If now $V_{n} = \amalg_{[n]\to [k]} M_{k}$ and $W_{n} \amalg_{[n]\to [k]} L_{k}$ then we can write
 \[
 U_{n}\times_{V_{n}}W_{n} = \amalg_{[n\to k_{1}], [n \to k_{2}], [n \to k_{3}]} N_{k_{1}}\times_{M_{k_{2}}}L_{k_{3}}.
 \]
  We define $K_{m} = \amalg^{*}_{[m\to k_{1}], [m\to k_{2}], [m \to k_{3}]} N_{k_{1}}\times_{M_{k_{2}}}L_{k_{3}}$ where the asterisk indicates that we leave out from the indexing set all pairs of maps which factor through a common $[m \to m-1]$. Then $\cat U\times_{\cat V}\cat W$  is split with free degeneracies $K_{n}$ in degree $n$.
  \end{proof}

 \begin{lemma}\label{lemma-goodhypercovers}
 There is a weak equivalence $Tw_{perf}(X,\mathcal R, \mathcal U) \simeq Tw_{perf}(X, \mathcal R)$ for any hypercover
  $\mathcal U$  in $HC$. In particular $Tw_{perf}(X, \cat R) \simeq Mod(\cat R(\cat X))^{cf}_{perf}$ if $X$ itself is in $\mathcal P$. Here the right hand side is the dg-category of fibrant and cofibrant complexes of $\cat R(X)$-modules which are perfect.
 \end{lemma}
 \begin{proof}
 Let us fix $\cat U$. By Lemma \ref{lemma-simplifykan} it suffices to compare $Tw(X, \cat R, \cat U)$ with $\hocolim_{HC^{op}}Tw_{perf}(X, \cat R, \cat W)$.

 We consider the category $HC_{\cat U}$ of all hypercovers in $HC$ that are refinements of $\cat U$.
 First, we examine the inclusion $\iota: HC_{\cat U} \to HC$. We claim that $HC_{\mathcal U}^{op}$ is homotopy terminal in $HC^{op}$, i.e.\ that $\iota$ is homotopy initial. We have to show that for any $\mathcal V \in HC$ the overcategory $(\iota \downarrow \mathcal V)$ is nonempty and contractible.

 This follows since it has a terminal object given by $\mathcal U\times_{X}\mathcal V$.
 This clearly satisfies the universal property, we just have to show it is in $HC$. Note that the fiber product is given by $U_{n}\times_{X}V_{n}$ in degree $n$. It is a hypercover by Lemma 7.2 in \cite{Stacks13}
  and is easily seen to be locally finite. Moreover, in every degree it is a coproduct of finite intersections of objects in $\cat P$, this $\cat U \times_{X} \cat V$ is a $\mathcal P$-hypercover.
 Finally, the fiber product is split by the argument from the previous lemma.

  By Lemma \ref{lemma-duggersthm-2} we now have
 \[
 \hocolim_{\cat V \in HC_{\cat U}^{op}} Tw_{perf}(X, \cat R, \cat V) \simeq \hocolim_{\cat W \in HC^{op}} Tw_{perf}(X, \cat R, \cat W) \simeq Tw_{perf}(X, \cat R).
 \]

 Secondly, we consider the inclusion of the 1-object category $\{\cat U \}$ in $HC_{\cat U}^{op}$. We claim that $Tw_{perf}(X, \cat R, \cat U) \simeq Tw_{perf}(X, \cat R, \cat V)$ for any $\cat V \in HC_{\cat U}^{op}$.

By Theorem 3.15 of \cite{ZhWei20151} we know that both sides are enhancements of $D_{perf}(QCoh(X))$, i.e.\ their homotopy categories agree.
But the results of that paper readily imply a stronger result. The functor $\cat S$ defined in Definition 3.2 of \cite{ZhWei20151} is a dg-functor form $Tw_{perf}(X, \cat R, \cat U)$ to $QCoh(X, \cat R)$. The right hand side is a model category and we may compose $\cat S$ with functorial fibrant cofibrant replacement. We claim this functor gives a DK equivalence $Tw_{perf}(X, \cat R, \cat U) \to QCoh(X,\cat R)^{cf}_{perf}$. Here the homotopy category of the right hand side is the derived category. That the functor lands in perfect complexes and is quasi-essentially surjective is immediate from the results of \cite{ZhWei20151}. Since the functor is compatible with shifts and induces isomorphisms on $H^{0}$ it induces isomorphisms on all the cohomology groups of the Hom-complexes. This shows it is quasi-fully faithful.
  Thus $Tw_{perf}(X, \cat R, \cat U) \simeq Tw_{perf}(X, \cat R, \cat V)$ as claimed.

 But then by Lemma \ref{lemma-duggerstheorem} applied to the inclusion of $\{\mathcal U\}$ in $HC^{op}_{\cat U}$ we have
 \[
 Tw_{perf}(X, \cat R, \cat U) \simeq \hocolim_{\cat V \in HC_{\cat U}^{op}}Tw_{perf}(X, \cat R, \cat V).\]
 Putting this together with the first weak equivalence completes the proof.

 For the special case we just consider the nerve of the trivial cover $X \to X$.
 \end{proof}

Moreover, given any $f: X \to Y$ the induced map $Tw_{perf}(Y) \to Tw_{perf}(X)$ may be represented by the natural restriction $f^{*}: Tw_{perf}(Y,\mathcal V) \to Tw_{perf}(X, \mathcal U)$ for any pair of good hypercovers compatible with $f$. (The universal property of the left Kan extension implies that there is natural transformation $Tw_{perf}(X, \mathcal U) \to Tw_{perf}(X)$ in the homotopy category.)

We now turn our attention to showing that $Tw_{perf}$ is a stack.
Hypercovers allow us to define the correct higher sheaf condition for presheaves with values in a homotopy category.
\begin{defi}\label{defi: stack}
A presheaf $F$ of dg-categories on a site $(\cat C, \tau)$ is a \emph{stack} or \emph{hypersheaf} if it satisfies the following condition:
$F(W) \simeq \holim_{i} F(U_{i})$ for any hypercover $\mathcal U = \{U_{i}\}$ of $W$. We also say that $F$ \emph{satisfies descent for hypercovers}.
\end{defi}
\begin{rmk}
Stacks are just the fibrant objects in the model category obtained by localizing at hypercovers, see \cite{dugger2004hypercovers}.
\end{rmk}

Given a presheaf of categories this condition asserts that the global data is glued from local data. Thus it is very important that categories of sheaves on a space satisfy this condition. We will show here directly that twisted perfect complexes do form a stack.
\begin{rmk}\label{rmk: cpx is a stack}
Theorem 2.8 of \cite{simpson05geometricity} asserts that the functor $Cpx^{cf}$ is a stack, where $Cpx^{cf}(U)$ is the subcategory of fibrant and cofibrant objects in $Cpx(U)$, and a model for the $(\infty,1)$-category of sheaves of complexes.
In the algebraic geometric setting this has been shown for example in \cite{Toen05a} or, in great generality, in the language of $\infty$-categories in \cite{Lurie11c, Lurie11g, Gaitsgory14}.
 \end{rmk}

A collection $S$ of hypercovers is \emph{dense} if every hypercover can be refined by a hypercover in $S$, and we now show that the hypercovers in $HC(X)$ are dense.
\begin{lemma}\label{lemma: existence of p-good refinement of hypercover}
If $X$ is a hereditarily paracompact topological space with a basis $\cat P$ then any hypercover in the site of open subsets of $X$ has a split, locally finite $\mathcal P$-refinement.
\end{lemma}
\begin{proof}
This follows from the proof of Theorem 8.6 in \cite{dugger2004hypercovers}, which shows the existence of a split refinement, and we just need to make the refinement be a locally finite $\mathcal P$-hypercover as well.

 In more details let $\mathcal U\to X$ be a hypercover and we know $U_0\to X$ is a cover. Since covers in a Verdier site are generated by a basis, there exists a family $W_a\to X$ and $\coprod_a W_a\to X$ refines $U_0\to X$. As $X$ is paracompact we can further obtain a locally finite $\cat P$-refinement of $\coprod_a W_a\to X$ and we denote it by $V_0$. The rest of the proof is the same as the induction procedure in \cite{dugger2004hypercovers} Theorem 8.6 and the refinement as above.
\end{proof}
\begin{coro}\label{corollary: existence of common refinements}
Moreover, any two hypercovers $U$ and $V$ have a common refinement as above.
\end{coro}
\begin{proof}
Given two hypercovers $U$ and $V$ it suffices to pick a refinement as in Lemma \ref{lemma: existence of p-good refinement of hypercover} for the fiber product $\cat U \times_{X} \cat V$.
\end{proof}

The crucial lemma now is the following:

\begin{lemma}\label{lemma-localiseatdense}
Consider a category of presheaves with values in a model category. The homotopy category obtained by localizing at all hypercovers is equivalent to localizing at a dense set of hypercovers. In particular, it suffices to check the stack condition in Definition \ref{defi: stack} on any dense set of hypercovers.
\end{lemma}
\begin{proof}
In the case of simplicial presheaves this is Theorem 6.2 in \cite{dugger2004hypercovers}.  A general argument can be found in Proposition 4.5.13 in \cite{Ohrt14}, we repeat it here for the reader's convenience.
Let $H$ denote the class of hypercovers and $H'$ be a dense subset. It suffices to show that any $H$-local weak equivalence is also an $H'$-local weak equivalence. This follows if any map in $H$ is $H'$-local.

Let $f: X \to Y$ be in $H$. As $H'$ is dense there exists $g: Z \to Y$ in $H'$ such that $f \circ h = g$ for some $h$. We will show that $h$ is actually an $H'$-local equivalence. Then $f$ is too, by the 2-out-of-3 property.
By the 2-out-of-3 property $h$ must be in $H$, thus there is $g': Z' \to X$ such in $H'$ that $g' = h \circ h'$ for some $h'$. After localizing at $H'$ we know that $g = fh$ and $g' = hh'$ become equivalences, thus $h$ has a left and a right inverse after localizing and is an $H'$-equivalence.

It follows that $H$-local objects and $H'$-local objects agree, and these are the stacks with respect to hypercovers and with respect to our dense subset of hypercovers.
\end{proof}

  \begin{thm}
 Let $(\cat C, \tau)$ be as above.
 $Tw_{perf}$ is a stack of dg-categories on $(\cat C, \tau)$.
 \end{thm}
 \begin{proof}
 Fix some $X \in \cat C$ . By Lemma \ref{lemma-localiseatdense} it suffices to check the hypersheaf condition for hypercovers in $HC(X)$.  We choose a hypercover $\mathcal V = \{V_{\bullet}\}$ in $HC(X)$ and by Proposition \ref{prop: twisted perfect complexes are totalization} and Lemma \ref{lemma-goodhypercovers} we have
 \[
 Tw_{perf}(X, \cat R)\simeq Tw_{perf}(X, \cat R, \cat V) \simeq \holim_{i} Perf(\cat R(V_{i})) \simeq \holim_{i} Tw_{perf}(V_{i}, \cat R). \qedhere
 \]
 \end{proof}

\subsection{$G$-equivariant twisted complexes}\label{subsection: group action}
In this subsection we consider a group action on a space. Let $X$ be a scheme and $G$ be a   group scheme. We assume $G$ acts on $X$ from the right. Then we have the associated simplicial scheme $[X/G]_{\bullet}$ which is defined as follows
$$
[X/G]_k=X\times\underbrace{ G\times \ldots\times G}_{k}.
$$
Moreover the face and degeneracy maps in this description are given by the following formula
\begin{equation*}
\begin{split}
\partial_0&(x,g_1,\ldots, g_k)=(xg_1,g_2, \ldots,g_k),\\
\partial_i&(x,g_1,\ldots, g_k)=(x,g_1,\ldots,g_ig_{i+1},\ldots,g_k)\text{ if } 1\leq i\leq k-1,\\
\partial_k& (x,g_1,\ldots, g_k)=(x,g_1,\ldots, g_{k-1}),
\end{split}
\end{equation*}
and
$$
s_i(x,g_1,\ldots, g_k)=(x,g_1\ldots,g_i,e,g_{i+1},\ldots,g_k), ~0\leq i\leq k.
$$

Recall that a $G$-equivariant sheaf on $X$ (see \cite{bernstein1994equivariant} Section 0.2) is a pair $(E,\phi)$ where $E$ is a sheaf on $M$ and $\phi$ is an isomorphism
$$
\phi: \partial_0^*(E)\overset{\cong}{\to} \partial_1^*(E)
$$
satisfying the following two conditions
\begin{enumerate}
\item The cocycle condition $\partial_1^*\phi=\partial_2^*\phi\circ \partial_0^*\phi$;
\item $s_0^*\phi=id$.
\end{enumerate}
A morphism between $G$-equivariant sheaves $(E,\phi)\to (F,\psi)$ is a sheaf map $E\to F$ which commutes with $\phi$ and $\psi$.

Now we consider complexes of sheaves and want a higher version of $G$-equivariance.  One approach is to consider $[X/G]$ as a quotient \emph{stack} and in \cite{olsson2007sheaves} Olsson developed the theory of derived categories of sheaves on stacks. Note that he is considering cartesian sheaves on a simplicial scheme. The way to relate this to the homotopy limit point of view is to recall that a homotopy limit can also be computed via \emph{strictification}, see discussion and references in \cite{holstein2015moritalocallyconstant}.

 In this paper we take another approach, i.e. we develop a theory directly from the homotopy limit on the level of dg-categories.

First notice that the usual dg-presheaves $\mathcal{F}: \text{Scheme}^{op}\to \dgCat$  are actually  dg-pseudo-functors hence $\mathcal{F}([X/G]_{\bullet})$ is merely a pseudo-cosimplicial object in dgCat. For example $\mathcal{F}=Cpx$ or $Perf$. Therefore we need the rectification $\widetilde{\mathcal{F}}$ of $\mathcal{F}$ as in Section \ref{subsection: rectification}.

To state our result we need to define the following maps: For any $k\geq p$ let $\rho_{k,p}$ be the front face map
\begin{equation*}
\begin{split}
\rho_{k,p}: [X/G]_k\to& [X/G]_p\\
(x,g_1,\ldots, g_k)\mapsto& (x,g_1,\ldots g_p)
\end{split}
\end{equation*}
and $\tau_{k,p}$ be the back face map
\begin{equation*}
\begin{split}
\tau_{k,p}: [X/G]_k\to& [X/G]_p\\
(x,g_1,\ldots, g_k)\mapsto& (x\cdot g_1 \cdot\ldots \cdot g_{k-p},g_{k-p+1},\ldots, g_k)
\end{split}
\end{equation*}

\begin{prop}\label{prop: totalization group action on a space}
Let $X$, $G$ and $[X/G]_{\bullet}$ be as above. Then an object in $\text{Tot} \widetilde{\mathcal{F}}([X/G]_{\bullet})$ consists of a pair $(E,\phi)$ where $E$ is an object of $\widetilde{\mathcal{F}}([X/G]_0)=\widetilde{\mathcal{F}}(X)$ and $\phi$ is a collection of maps
$$
\phi^{k,1-k}\in Hom^{1-k}_{\widetilde{\mathcal{F}}([M/G]_k)}(\tau_{k,0}^*(E),\rho_{k,0}^*(E)), \text{ for all } k\geq 1
$$
which satisfy the following two conditions
\begin{enumerate}
\item the Maurer-Cartan equation
\begin{equation}\label{equation: MC for group action}
(-1)^kd(\phi^{k,1-k})+\sum_{j=1}^{k-1}(-1)^j\partial_j^*(\phi^{k-1,2-k})+\sum_{j=1}^{k-1}(-1)^{(1-j)(k-j)}\rho_{k,j}^*(\phi^{j,1-j})\tau_{k,k-j}^*(\phi^{k-j,1-k+j})=0;
\end{equation}

\item the non-degenerate condition: $\phi^{1,0}$ is invertible up to homotopy.
\end{enumerate}

Moreover, a morphism $\theta$ of degree $m$ from $(E,\phi)$ to $(F,\psi)$ is given by a collection
$$
\theta^{k,m-k}\in \Hom^{m-k}_{\widetilde{\mathcal{F}}([X/G]_k)}(\tau_{k,0}^*(E),\rho_{k,0}^*(F)) \text{ for all } k\geq 0
$$
and the differential is given by
$$
d\theta=\D\theta+\phi\cdot \theta-(-1)^m\theta\cdot \psi
$$
or more explicitly
\begin{equation*}
\begin{split}
(d\theta)^{k,m+1-k}&=(-1)^k d_{\widetilde{\mathcal{F}}(X_k)}\theta^{k,m-k}+\sum_{j=1}^{k-1}(-1)^j\partial^*_j\theta^{k-1,m+1-k}\\
+\sum_{l=1}^k&(-1)^{(1-l)(k-l)}\rho_{k,l}^*\phi^{l,1-l}\tau_{k,k-l}^*\theta^{k-l,m-k+1}+\sum_{l=0}^{k-1}(-1)^{(m-l)(k-l)}\rho_{k,l}^*\theta^{l,m-l}\tau_{k,k-l}^*\psi^{k-l,1-k+l}.
\end{split}
\end{equation*}
\end{prop}
\begin{proof}
It is a direct corollary of Theorem \ref{thm: totalization in dgCat}.
\end{proof}

\begin{rmk}
Proposition \ref{prop: F tilde is weakly equivalent to F}  tells us that for any $X$, the dg-functor $\widetilde{\mathcal{F}}(X)\to\mathcal{F}(X)$ is quasi-essentially surjective and degreewise fully faithful. As a result in Proposition \ref{prop: totalization group action on a space} we can replace $E$ and $\phi$ by their image in $\mathcal{F}$ and the result will not be infected.
\end{rmk}

For $k=1$ Equation (\ref{equation: MC for group action}) is simply $d\phi^{1,0}=0$, i.e. $\phi^{1,0}\in\Hom^{0}_{\mathcal{F}([M/G]_1)}(\partial_0^*(E),\partial_1^*(E))$ is closed. For $k=2$  Equation (\ref{equation: MC for group action}) becomes
$$
d\phi^{2,-1}-\partial_1^*\phi^{1,0}+\rho_{2,1}^*\phi^{1,0}\tau_{2,1}^*\phi^{1,0}=0,
$$
which means that $\phi^{1,0}$ satisfies the cocycle condition \emph{up to homotopy} and $\phi^{2,-1}$ gives the homotopy operator.

Now we need to verify that  the cosimplicial diagram  $\mathcal{F}([X/G]_{\bullet})$ is  Reedy fibrant so that the totalization is weakly equivalent to the homotopy limit.

\begin{prop}
If the  $G$ is discrete and $\mathcal{F}$ sends finite coproducts to products, then the construction in Proposition \ref{prop: totalization group action on a space} gives  the homotopy limit of $\mathcal{F}([X/G]_{\bullet})$.
\end{prop}
\begin{proof}
  It is easy to see that in this case the simplicial space  $[X/G]_{\bullet}$ is split and hence by Proposition \ref{prop: split and left exact implies Reedy fibrant} $\mathcal{F}([X/G]_{\bullet})$ is  Reedy fibrant  and therefore the totalization and homotopy limit are weakly equivalent by Theorem \ref{thm: totalization and homotopy limit are weakly equi for fibrant object}.
\end{proof}

This proposition hold for example if we consider $\mathcal{F}=Cpx$, the functor that sends $X$ to the dg-category of complexes of sheaves of $\mathcal{O}_X$-modules on $X$.

We are left to consider the case when $G$ is not discrete. In this case it seems that the diagram  $\mathcal{F}([X/G]_{\bullet})$ is  no longer Reedy fibrant. (The matching maps are not levelwise surjective on hom-complexes in general.) We still expect the totalization and the homotopy limit to agree.
\begin{conj}
If $G$ is a reductive group scheme and $\mathcal{F}=Cpx$, then the construction in Proposition \ref{prop: totalization group action on a space} still gives  the homotopy limit of $Cpx([X/G]_{\bullet})$.
\end{conj}

\bibliography{homotopylimitbib}{}

\bibliographystyle{alpha}

\end{document}